\documentclass[12pt,a4paper]{amsart} 
\usepackage[utf8]{inputenc}
\usepackage[english]{babel}

\usepackage{amsmath}
\usepackage{amsfonts}
\usepackage{amssymb}
\usepackage{color}

\usepackage{amsthm}

\newtheorem{theorem}{Theorem} [section]

\newtheorem{corollary}[theorem]{Corollary} 

\newtheorem{conjecture}{Conjecture}
\newtheorem{open}{Open Question}

\newcounter{defno}

\usepackage{graphics}
\usepackage{epsfig}

\usepackage{url}
\usepackage{hyperref}

\usepackage[a4paper,top=2.5cm,bottom=2.8cm,
left=3.2cm,right=3.2cm]{geometry}
\markleft{\hfill \textsc{More Relationships between Central Quad and Reference Quad} \hfill}
\long\def\void#1{}

\begin{document}
International Journal of  Computer Discovered Mathematics (IJCDM) \\
ISSN 2367-7775 \copyright IJCDM \\
Volume 10, 2025 pp. 91--123  \\
web: \url{http://www.journal-1.eu/} \\
Received 5 Feb. 2025. Published on-line 11 April 2025\\ 

\copyright The Author(s) This article is published 
with open access.\footnote{This article is distributed under the terms of the Creative Commons Attribution License which permits any use, distribution, and reproduction in any medium, provided the original author(s) and the source are credited.} \\
\bigskip
\bigskip

\begin{center}
	{\Large \textbf{More Relationships between a Central Quadrilateral and its Reference Quadrilateral}} \\
	\medskip
	\bigskip
        \bigskip

	\textsc{Stanley Rabinowitz$^a$ and Ercole Suppa$^b$} \\

	$^a$ 545 Elm St Unit 1,  Milford, New Hampshire 03055, USA \\
	e-mail: \href{mailto:stan.rabinowitz@comcast.net}{stan.rabinowitz@comcast.net}\footnote{Corresponding author} \\
	web: \url{http://www.StanleyRabinowitz.com/} \\
	
	$^b$ Via B. Croce 54, 64100 Teramo, Italia \\
	e-mail: \href{mailto:ercolesuppa@gmail.com}{ercolesuppa@gmail.com} \\
	web: \url{https://www.esuppa.it} \\

\bigskip
\bigskip

\end{center}
\bigskip

\textbf{Abstract.}
The diagonals of a quadrilateral form four associated triangles, called \emph{half triangles}.
Each half triangle is bounded by two sides of the quadrilateral and one diagonal.
If we locate a triangle center (such as the incenter, centroid, orthocenter, etc.) in each of these triangles, the four triangle
centers form another quadrilateral called a \emph{central quadrilateral}.
For each of various shaped quadrilaterals, and each of 1000 different triangle
centers, we compare the reference quadrilateral to the central quadrilateral.
Using a computer, we determine how the two quadrilaterals are related.
For example, we test to see if the two quadrilaterals are congruent, similar, have the same area,
or have the same perimeter.



\newcommand{\ru}{\rule[-8pt]{0pt}{20pt}}

\newcommand{\m}{\mathrm{m}}
\newcommand{\h}{\mathrm{h}}
\newcommand{\diag}{\mathrm{dp}}
\newcommand{\oo}{\mathrm{o}}
\newcommand{\ho}{\mathrm{homot}}
\newcommand{\persp}{\mathrm{persp}}
\newcommand{\ortho}{\mathrm{ortho}}
\newcommand{\conic}{\mathrm{conic}}
\newcommand{\hyperb}{\mathrm{hyperb}}
\newcommand{\ponce}{\mathrm{ponce}}
\newcommand{\stein}{\mathrm{stein}}
\newcommand{\anti}{\mathrm{anti}}
\newcommand{\centro}{\mathrm{centro}}
\newcommand{\incent}{\mathrm{i}}

\newcommand{\no}[1]{\textcolor{red}{#1}}

\newenvironment{code}[2]
{
\medskip
\hspace{#1}%
\begin{minipage}{#2}
\color{blue}
}
{
\color{black}
\smallskip
\end{minipage}%
}


\bigskip
\bigskip
%
\section{Introduction}
\label{section:introduction}

The diagonals of a quadrilateral (called the \emph{reference quadrilateral}) form four associated triangles, called \emph{half triangles},
shown in Figure~\ref{fig:sideTriangles}.
Each half triangle is bounded by two sides of the quadrilateral and one diagonal.
The reference quadrilateral is always named $ABCD$.
The four triangles (numbered 1 to 4) are shown
in Figure \ref{fig:sideTriangles}.

\begin{figure}[h!t]
\centering
\includegraphics[width=0.6\linewidth]{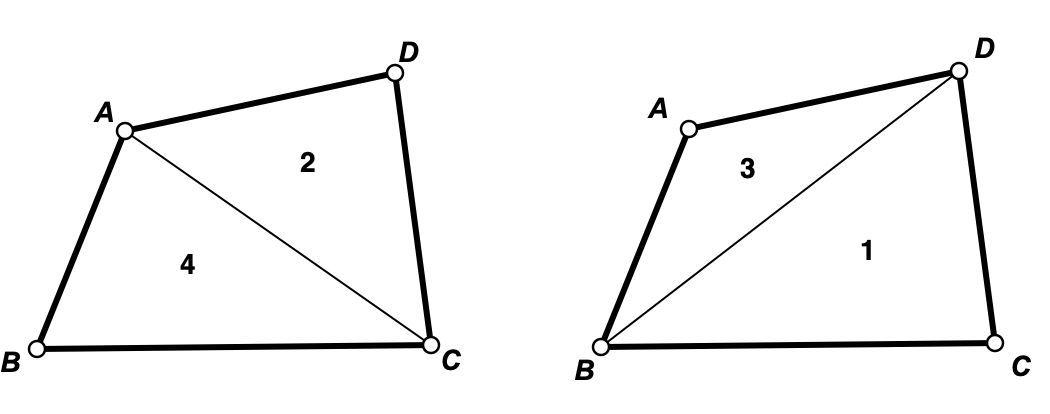}
\caption{Half Triangles}
\label{fig:sideTriangles}
\end{figure}

The triangles have been numbered so that triangle 1 is opposite vertex $A$,
triangle 2 is opposite vertex $B$, etc.
The four triangles are $\triangle BCD$, $\triangle ACD$, $\triangle ABD$, and $\triangle ABC$.

Triangle centers are selected in each triangle (for example, incenters, centroids, or orthocenters).
The same type of triangle center is used with each half triangle.
In order, the names of these points are $E$, $F$, $G$, and $H$, as shown in Figure \ref{fig:centralQuadrilateral}.
These four centers form a quadrilateral $EFGH$ that will be called the \emph{central quadrilateral}.
Quadrilateral $EFGH$ need not be convex.

\begin{figure}[h!t]
\centering
\includegraphics[width=0.4\linewidth]{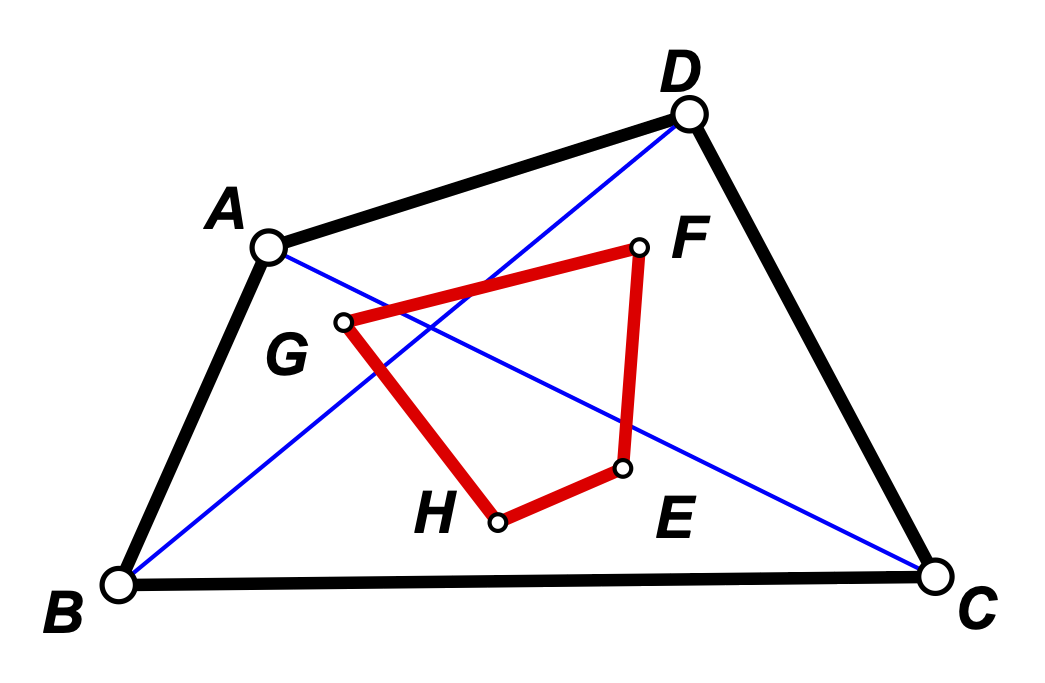}
\caption{Central Quadrilateral}
\label{fig:centralQuadrilateral}
\end{figure}

The purpose of this paper is to determine interesting relationships between a reference quadrilateral
and its central quadrilateral. This paper extends our previous results found in \cite{relationships}.

\newpage

\section{Types of Quadrilaterals Studied}
\label{section:quadrilaterals}

We are only interested in convex reference quadrilaterals that have a certain amount of symmetry.
For example, we excluded bilateral quadrilaterals (those with two equal sides),
bisect-diagonal quadrilaterals (where one diagonal bisects another), right kites,
right trapezoids, and golden rectangles.
The types of quadrilaterals we studied are shown in Table \ref{table:quadrilaterals}.
The sides of the quadrilateral, in order, have lengths $a$, $b$, $c$, and $d$.
The diagonals have lengths $p$ and $q$.
The measures of the angles of the quadrilateral, in order, are $A$, $B$, $C$, and $D$.

\begin{table}[ht!]
\caption{}
\label{table:quadrilaterals}
\begin{center}
\footnotesize
\begin{tabular}{|l|l|l|}\hline
\multicolumn{3}{|c|}{\textbf{\color{blue}\Large \strut Types of Quadrilaterals Considered}}\\ \hline
\textbf{Quadrilateral Type}&\textbf{Geometric Definition}&\textbf{Algebraic Condition}\\ \hline
general&convex&none\\ \hline
cyclic&has a circumcircle&$A+C=B+D$\\ \hline
tangential&has an incircle&$a+c=b+d$\\ \hline
extangential&has an excircle&$a+b=c+d$\\ \hline
parallelogram&opposite sides parallel&$a=c$, $b=d$\\ \hline
equalProdOpp&product of opposite sides equal&$ac=bd$\\ \hline
equalProdAdj&product of adjacent sides equal&$ab=cd$\\ \hline
orthodiagonal&diagonals are perpendicular&$a^2+c^2=b^2+d^2$\\ \hline
equidiagonal&diagonals have the same length&$p=q$\\ \hline
Pythagorean&equal sum of squares, adjacent sides&$a^2+b^2=c^2+d^2$\\ \hline
kite&two pair adjacent equal sides&$a=b$, $c=d$\\ \hline
trapezoid&one pair of opposite sides parallel&$A+B=C+D$\\ \hline
rhombus&equilateral&$a=b=c=d$\\ \hline
rectangle&equiangular&$A=B=C=D$\\ \hline
Hjelmslev&two opposite right angles&$A=C=90^\circ$\\ \hline
isosceles trapezoid&trapezoid with two equal sides&$A=B$, $C=D$\\ \hline
APquad&sides in arithmetic progression&$d-c=c-b=b-a$\\ \hline
\end{tabular}
\end{center}
\end{table}

The following combinations of entries in the above list were also considered:
bicentric quadrilaterals (cyclic and tangential), exbicentric quadrilaterals (cyclic and extangential),
bicentric trapezoids, cyclic orthodiagonal quadrilaterals, equidiagonal kites,
equidiagonal orthodiagonal quadrilaterals, equidiagonal orthodiagonal trapezoids,
harmonic quadrilaterals (cyclic and equalProdOpp), orthodiagonal trapezoids, tangential trapezoids,
and squares (equiangular rhombi).

So, in addition to the general convex quadrilateral, a total of 27 other types of quadrilaterals
were considered in this study.

A graph of the types of quadrilaterals considered is shown in Figure \ref{fig:quadShapes}.
An arrow from A to B means that any quadrilateral of type B is also of type A.
For example: all squares are rectangles and all kites are orthodiagonal.
If a directed path leads from a quadrilateral of type A to a quadrilateral of type B, then we will
say that A is an \emph{ancestor} of B. For example, an equidiagonal quadrilateral is an ancestor of a rectangle.
In other words, all rectangles are equidiagonal.

\begin{figure}[h!t]
\centering
\scalebox{1}[1.5]{\includegraphics[width=1\linewidth]{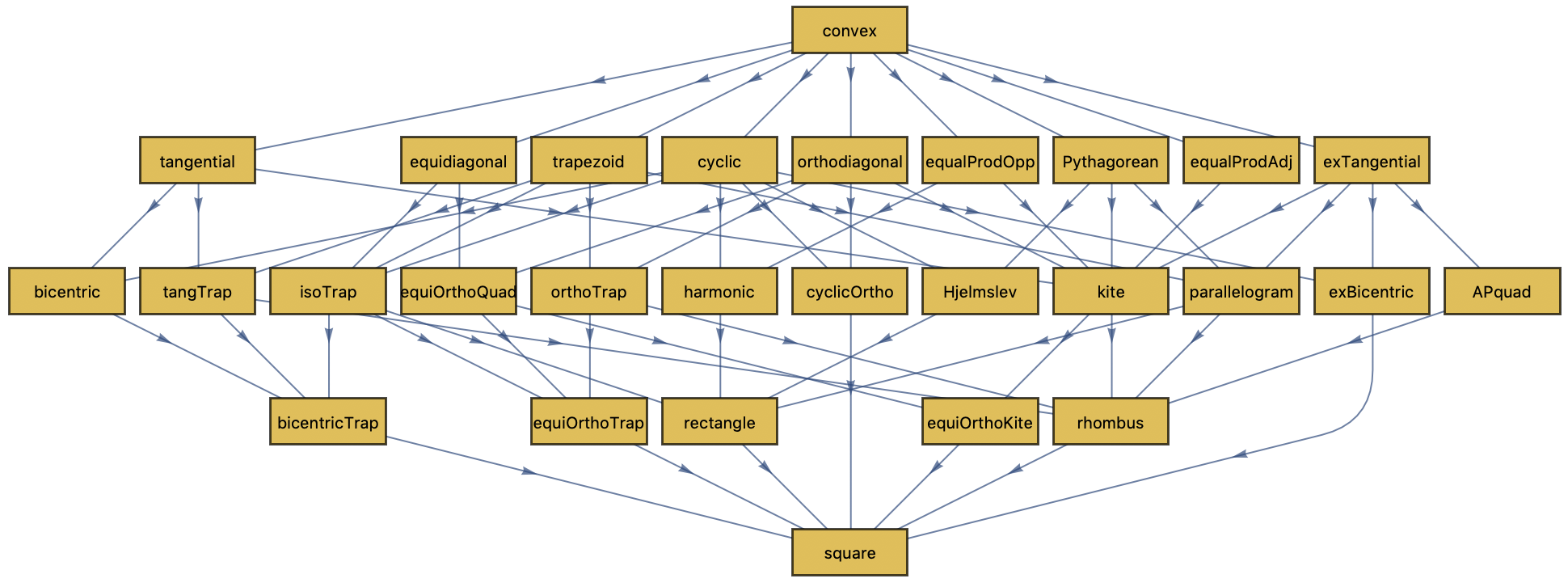}}
\caption{Quadrilateral Shapes}
\label{fig:quadShapes}
\end{figure}

Unless otherwise specified, when we give a theorem or table of properties of a quadrilateral, we will omit an entry
for a particular shape quadrilateral if the property is known to be true for an ancestor of that quadrilateral.

\newpage

\section{Centers}
\label{section:centers}

In this study, we will place triangle centers in the four half
triangles. We use Clark Kimberling's definition of a triangle center \cite{KimberlingA}.

A \emph{center function} is a nonzero function $f(a,b,c)$
homogeneous in $a$, $b$, and $c$ and symmetric in $b$ and $c$.
\emph{Homogeneous} in $a$, $b$, and $c$ means that
$$f(ta,tb,tc)=t^nf(a,b,c)$$
for some nonnegative integer $n$,
all $t>0$, and all positive real numbers $(a,b,c)$ satisfying $a<b+c$, $b<c+a$, and $c<a+b$.
\emph{Symmetric} in $b$ and $c$ means that
$$f(a,c,b)=f(a,b,c)$$
for all $a$, $b$, and $c$.

A \emph{triangle center} is an equivalence class $x:y:z$ of ordered triples $(x,y,z)$
given by
$$x=f(a,b,c),\quad y=f(b,c,a),\quad z=f(c,a,b).$$

Tens of thousands of interesting triangle centers have been cataloged in the Encyclopedia of Triangle Centers \cite{ETC}. We use $X_n$ to denote the $n$-th named center in this encyclopedia.

Note that if the center function of a certain center is $f(a,b,c)$,
then the trilinear coordinates of that point with respect to a triangle with sides $a$, $b$, and $c$ are
$$\Bigl(f(a,b,c):f(b,c,a):f(c,a,b)\Bigr).$$
The barycentric coordinates for that point would then be
$$\Bigl(af(a,b,c):bf(b,c,a):cf(c,a,b)\Bigr).$$

\newpage

\section{Methodology}
\label{section:methodology}

We used a computer program called GeometricExplorer to compare quadrilaterals
with their central quadrilateral. Starting with each type of quadrilateral listed in
Figure~\ref{fig:quadShapes} for the reference quadrilateral, we placed triangle centers
in each of the four half triangles.

For each $n$ from 1 to 1000, we determined center $X_n$
of each of the half triangles of the reference quadrilateral.
The program then analyzes the central quadrilateral formed by these four centers
and reports if the central quadrilateral is related to the reference quadrilateral.
Points at infinity were omitted.
The types of relationships checked for are shown in Table \ref{table:relationships}.

\label{table:relationships}
\begin{center}
\begin{tabular}{|l|p{3.5in}|}
\hline
\multicolumn{2}{|c|}{\Large \strut \textbf{\color{blue}Relationships Checked For}}\\
\hline
\textbf{notation}&\textbf{description}\\ \hline
\ru $[ABCD]=[EFGH]$&the~quadrilaterals~have~the~same~area \tiny (This relationship is excluded if the quadrilaterals are congruent.)\\ \hline
\ru $[ABCD]=k[EFGH]$&the area of $ABCD$ is $k$ times the area of $EFGH$\,$\dagger$\\ \hline
\ru $ABCD\cong EFGH$&the quadrilaterals are congruent\\ \hline
\ru $ABCD\sim EFGH$&the quadrilaterals are similar \tiny (This relationship is excluded if the quadrilaterals are homothetic.)\\ \hline
\ru $\partial ABCD=\partial EFGH$&the quadrilaterals have the same perimeter \tiny (This relationship is excluded if the quadrilaterals are congruent.)\\ \hline
\ru $\odot ABCD\cong\odot EFGH$&the quadrilaterals have congruent circumcircles \tiny (This relationship is excluded if the quadrilaterals are congruent.)\\ \hline
\ru $\odot ABCD\equiv \odot EFGH$&the quadrilaterals have the same circumcircle\\ \hline
\ru $\mathrm{o}(ABCD)=\mathrm{o}(EFGH)$&the quadrilaterals have the same circumcenter
\tiny (This relationship is excluded if the quadrilaterals have the same circumcircle.)\\ \hline
\ru $\mathrm{i}(ABCD)=\mathrm{i}(EFGH)$&the quadrilaterals have the same incenter\\ \hline
\ru $\mathrm{dp}(ABCD)=\mathrm{dp}(EFGH)$&the quadrilaterals have the same diagonal point\\ \hline
\ru $\mathrm{persp}(ABCD,EFGH)$&the quadrilaterals are perspective
\tiny (This relationship is excluded if the quadrilaterals are homothetic.)\\ \hline
\ru $\mathrm{homot}(ABCD,EFGH)$&the quadrilaterals are homothetic\\ \hline
\ru $\mathrm{conic}(ABCD,EFGH)$&the quadrilaterals have a common noncircular circumconic\\ \hline
\ru $\mathrm{hyperb}(ABCD,EFGH)$&the quadrilaterals have a common circumconic which is a rectangular hyperbola.
\tiny (By definition, the center of this hyperbola is the Poncelet point (QA-P2) of both $ABCD$ and $EFGH$.)\\ \hline
\ru $\mathrm{ctr1}[ABCD]=\mathrm{ctr2}[EFGH]$&the quadrilaterals have coincident centers\\ \hline
\multicolumn{2}{|l|}{$\dagger$ \small Only rational values of $k$ were checked for with denominators less than 10.}\\
\hline
\end{tabular}
\end{center}

\newpage

The types of quadrilateral centers considered are shown in Table \ref{table:centers}.
For example, the relationship $\ponce[ABCD]=\stein[EFGH]$ means that the
Poncelet point of quadrilateral $ABCD$ coincides with the Steiner point of quadrilateral $EFGH$.

\begin{table}[ht!]
\caption{}
\label{table:centers}
\begin{center}
\begin{tabular}{|l|l|l|}
\hline
\multicolumn{3}{|c|}{\Large \strut \textbf{\color{blue}Quadrilateral Centers Considered}}\\
\hline
\textbf{name}&\textbf{description}&\textbf{symbol}\\ \hline
vertex centroid&(QA-P1)&$\m$\\ \hline
Poncelet point&(QA--P2), also known as the Euler-Poncelet point&$\ponce$\\ \hline
Steiner point&(QA--P3), also known as the Gergonne-Steiner point&$\stein$\\ \hline
diagonal point&intersection of the diagonals (QG--P1)&$\diag$\\ \hline
\multicolumn{3}{|c|}{\small \color{blue}The following centers are only defined for cyclic quadrilaterals.}\\
\hline
anticenter&intersection of the maltitudes (QA--P2)&$\anti$\\ \hline
circumcenter&center of circumscribed circle (QA-P3)&$\oo$\\ \hline
centrocenter&\small center of circle through centroids of half triangles (QA-P7)&$\centro$\\ \hline
orthocenter&center of circle through orthocenters of half triangles&$\h$\\ \hline
\multicolumn{3}{|c|}{\small \color{blue}The following centers are only defined for tangential quadrilaterals.}\\
\hline
incenter&center of inscribed circle&i\\ \hline
\void{
quasi centroid$\dagger$&(QG-P4)\\ \hline
quasi circumcenter$\dagger$&(QG-P5)\\ \hline
quasi orthocenter$\dagger$&(QG-P6)\\ \hline
quasi nine-point center$\dagger$&(QG-P7)\\ \hline
quasi incenter$\dagger$&\\ \hline
Kirikami center$\dagger$&(QG-P15)\\ \hline
Miquel point$\dagger$&(QL-P1)\\ \hline
\multicolumn{3}{|l|}{$\dagger$ \small This point might be omitted from the final paper.}\\ \hline
}
\end{tabular}
\end{center}
\end{table}

Some quadrilateral centers only exist for certain shape quadrilaterals.
For example, the circumcenter, anticenter, orthocenter, and centrocenter only apply to cyclic quadrilaterals.
The incenter only applies to tangential quadrilaterals.
A code in parentheses represents the name for the point as listed in
the Encyclopedia of Quadri-Figures~\cite{EQF}.

When reporting perspectivities or homotheties, we will specify what type of point the perspector is.
Only quadrangle points listed in \cite{EQF} are detected. As of January 2025, only 44 points were listed.

For example, the property
QA-Pi=$\persp(ABCD,EFGH)$=QA-Pj means that the perspector is the QA-Pi point of quadrilateral $ABCD$
and is the QA-Pj point of quadrilateral $EFGH$.

A similar notation is used when describing properties involving conics. We check to see if the center of the conic
is one of the known quadrangle points. For example, the property
QA-Pi=$\conic(ABCD,EFGH)$ means that the center of the common circumconic is the QA-Pi point of quadrilateral $ABCD$.

\newpage

The most common quadrangle points are described in the following table. See \cite{EQF} for the definition of terms and more details.
Note that a \emph{quadrangle} is an unordered set of four points in the plane (no three of which are collinear).

\begin{table}[ht!]
\caption{}
\label{table:quadcenters}
\begin{center}
\begin{tabular}{|l|p{2in}|p{2.83in}|}
\hline
\multicolumn{3}{|c|}{\Large \strut \textbf{\color{blue}Common Quadrangle Centers}}\\
\hline
\textbf{symbol}&\textbf{common name}&\textbf{description}\\ \hline
QA-P1&Quadrangle Centroid&center of gravity of equal masses placed at the vertices\\ \hline
QA-P2&Euler-Poncelet Point&common point of the nine-point circles of the half triangles\\ \hline
QA-P3&Gergonne-Steiner Point&Common point of the four midray circles\\ \hline
QA-P4&Isogonal Center&homothetic center of $ABCD$ with the 2nd generation isogonal conjugate quadrangle\\ \hline
QA-P5&Isotomic Center&perspector of $ABCD$ with the isotomic conjugate quadrangle\\ \hline
QA-P6&Parabola Axes Crosspoint&intersection point of the axes of the two parabolas that can be constructed through $A$, $B$, $C$, and $D$\\ \hline
QA-P7&9-pt Homothetic Center&homothetic center of $ABCD$ with the quadrangle composed of four 2nd generation nine-point venters\\ \hline
QA-P9&QA Miquel Center&common point of the three Miquel circles of the half triangles\\ \hline
QA-P12&Orthocenter of the Diagonal Triangle&orthocenter of the triangle formed by the three diagonal points of $ABCD$\\ \hline
QA-P34&Euler-Poncelet Point of the Centroid Quadrangle&Euler-Poncelet point of the quadrangle formed by centroids of the half triangles\\ \hline
\end{tabular}
\end{center}
\end{table}

\newpage

\section{Barycentric Coordinates and Quadrilaterals}

The program we used to find results about central quadrilaterals (GeometricExplorer) is a useful
tool for discovering results, but it does not prove that these results are true.
GeometricExplorer uses numerical coordinates (to 15 digits of precision) for locating
all the points. Thus, a relationship found by this program does not constitute a proof that the result is correct,
but gives us compelling evidence for the validity of the result.

If a theorem in this paper is accompanied by a figure, this means that the
figure was drawn using either Geometer's Sketchpad or GeoGebra.
In either case, we used the drawing program to dynamically vary the points
in the figure. Noticing that the result remains true as the points vary offers
further evidence that the theorem is true. But again, this does not constitute a proof.

To prove the results that we have discovered, we use geometric methods, when possible.
If we could not find a purely geometrical proof, we turned to analytic methods using
barycentric coordinates and performing exact symbolic computation using Mathematica
and the package \texttt{baricentricas.m}\footnote{The package \texttt{baricentricas.m} written by F.~J.~G~Capit\'an can be freely downloaded from \url{http://garciacapitan.epizy.com/baricentricas/}}.

These analytic proofs are given in Mathematica Notebooks included with the supplementary material accompanying the on-line publication of this paper.

If our only ``proof'' of a particular relationship is by using numerical calculations (and not using exact computation),
then we have colored the center \no{red} in the table of relationships.

When proving results analytically, we used barycentric coordinates. We assume the reader is familiar with this coordinate system.
Given a quadrilateral $ABCD$, we set up a barycentric coordinate system using $\triangle ABC$ as the reference triangle.
We assign coordinates $(p:q:r)$ to point $D$ as shown in Figure~\ref{fig:baryCoordinatespqr}.
Note that $AB=c$, $BC=a$, and $AC=b$.

\begin{figure}[h!t]
\centering
\includegraphics[width=0.35\linewidth]{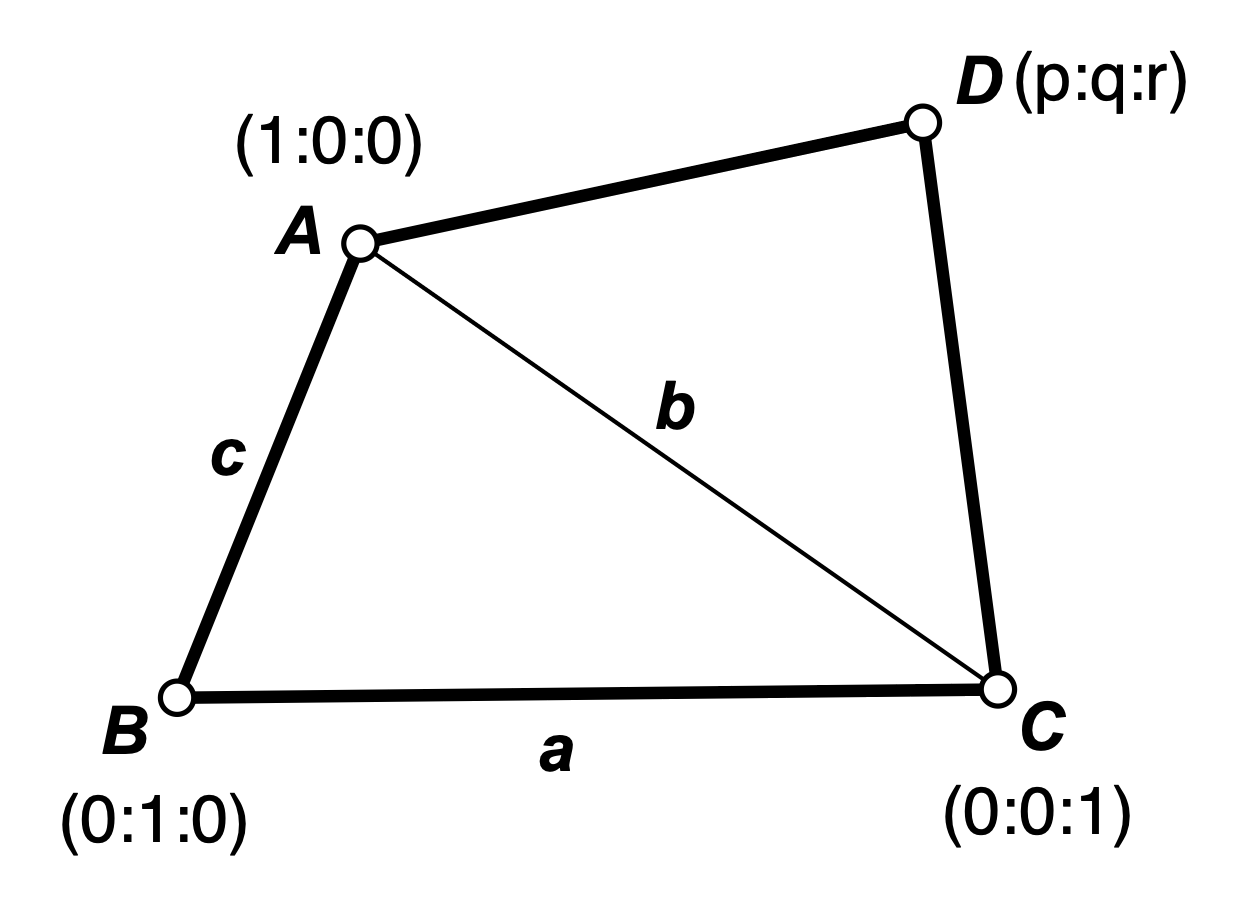}
\caption{barycentric coordinate system for quadrilateral $ABCD$}
\label{fig:baryCoordinatespqr}
\end{figure}

The barycentric coordinates for the various triangle centers were found from \cite{ETC}.
To find the coordinates of a center $(u:v:w)$ with respect to a triangle $XYZ$,
we use the function \texttt{CentroETCTriangulo} in the \texttt{baricentricas.m} package
via the call \texttt{CentroETCTriangulo[\{u,v,w\},\{ptX,ptY,ptZ\}]} where \texttt{ptX}, \texttt{ptY}, and \texttt{ptZ}, 
are the barycentric coordinates for the vertices of $\triangle XYZ$.

\newpage

When analyzing an initial quadrilateral with a special shape, we restrict the values of $p$, $q$, and $r$
by specifying a condition that $a$, $b$, $c$, $p$, $q$, and $r$ must satisfy using the conditions shown in the following table.

\begin{center}
\begin{tabular}{|p{2in}|p{3.64in}|} \hline
		\textbf{Geometrical condition}        & \textbf{Analytic condition}\\\hline
\ru             $A,B,C,D$ concyclic                  &$a^2 q r+ b^2 p r +c^2 p q =0 $  \\ \hline
\ru		$AB+CD=BC+AD$            &  $2 p r \left(a^2-2 a c+b^2+c^2\right)+p^2 (a-b-c) (a+b-c)+r^2 (a-b-c) (a+b-c)-4 c p q (a-c)+4 a q r (a-c)=0$\\ \hline
\ru		$AB+BC=AD+DC$        &  \scriptsize{$p^2 \left(a^2+2 a c-b^2+c^2\right)+p r \left(2 a^2+4 a c+2 b^2+2 c^2\right)+r^2 \left(a^2+2 a c-b^2+c^2\right)+q r \left(4 a^2+4 a c\right)+p q \left(4 a c+4 c^2\right)=0$}\\ \hline
\ru		$BC+CD=AB+AD$        &  \scriptsize{$p^2 \left(a^2-2 a c-b^2+c^2\right)+p r \left(2 a^2-4 a c+2 b^2+2 c^2\right)+r^2 \left(a^2-2 a c-b^2+c^2\right)+q r \left(4 a^2-4 a c\right)+p q \left(4 c^2-4 a c\right)=0$} \\ \hline
\ru		$AB=BC$             &$a=c$  \\ \hline
\ru		$AD\parallel BC$             &$q+r=0$  \\ \hline
\ru		$AB\parallel CD$             &$p+q=0$  \\ \hline
\ru		$AB\cdot CD=AC\cdot DA$           & \scriptsize$p(b^2-c^2)(qc^2+rb^2)+a^2\left(c^2q(p+q)-b^2r(p+r)\right)=0$ \\ \hline
\ru		$AB\cdot AC=BC\cdot CD$         & $a^4q(p+q)+a^2p\left(b^2(p+q)-qc^2\right)=b^2c^2(p+q+r)^2$ \\ \hline
\ru		$AC\perp BD$         &$b^2(p-r)=(a^2-c^2)(p+r)$ \\ \hline
\ru		$AB\perp BC$         &$b^2=a^2+c^2$ \\ \hline
\ru		$AC=BD$          &  $b^2\left(p^2+(q+r)^2+p(2q+3r)\right)=(p+r)(c^2p+a^2r)$\\ \hline
\ru		$AB^2+AC^2=BC^2+CD^2$           & \scriptsize${a^2 \left(p^2+p (q+2 r)+r (2 q+r)\right)=}\ b^2 \left(2 p^2+p (3 q+2
   r)+(q+r)^2\right) +c^2 \left(p^2+p (q+2 r)+(q+r)^2\right)$ \\ \hline
\ru		kite                   &$p+2q+r=0 \;\wedge\; b^2+q=0\;\wedge\; c^2=a^2+q+r$  \\ \hline
\ru		parallelogram          &$q+r=0\;\wedge\;p+q=0$\\ \hline
\ru		rhombus                 & $q+r=0\;\wedge\;p+q=0\;\wedge\; a=c$  \\ \hline
\ru		rectangle              & $q+r=0\;\wedge\;p+q=0\;\wedge\; b^2=a^2+c^2$ \\ \hline
\ru		isosceles trapezoid    &$b^2p+\left(a^2-c^2\right)q=0$  \\ \hline
\ru            harmonic quadrilateral& $ra^2=pc^2\;\wedge\; pb^2+2qa^2=0$ \\ \hline
\ru            orthodiagonal quadrilateral& $a^2 (p+r)+b^2 (r-p)-c^2 (p+r)=0$\\ \hline

\end{tabular}
\end{center}

\bigskip
If a quadrilateral shape is formed by a combination of conditions, then the condition used to obtain that shape is the conjunction
of the primitive conditions.

For example, a parallelogram has $AD\parallel BC$ and $AB\parallel CD$,
so the condition on $p$, $q$, and $r$ that makes $ABCD$ a parallelogram is $(q+r=0)\wedge (p+q=0)$.
A rhombus is a parallelogram with the added condition $AB=BC$, so we add in the analytical condition $a=c$.


When checking to see if a point is a notable center of quadrilateral $ABCD$, we use the CT coordinates found from \cite{EQF}.
The coordinates were scaled to get rid of any fractions. Otherwise, applying these coordinates to certain shape quadrilaterals would produce divide-by-zero errors.

\newpage
\subsection{Example}\ \\

We now give an example of how barycentric coordinates can be used to prove the results in this paper.
We show how to prove the following theorem using barycentric coordinates.

\begin{theorem}
\label{thm:odX5}
Let $ABCD$ be an orthodiagonal quadrilateral.
Let $E$, $F$, $G$, and $H$ be the $X_{5}$ -points of $\triangle BCD$, $\triangle CDA$, 
$\triangle DAB$, and $\triangle ABC$, respectively.
Then the centroid of quadrilateral $ABCD$ coincides with the diagonal point of quadrilateral $EFGH$ (Figure~\ref{fig:odX5}).
\end{theorem}

\begin{figure}[h!t]
\centering
\includegraphics[width=0.4\linewidth]{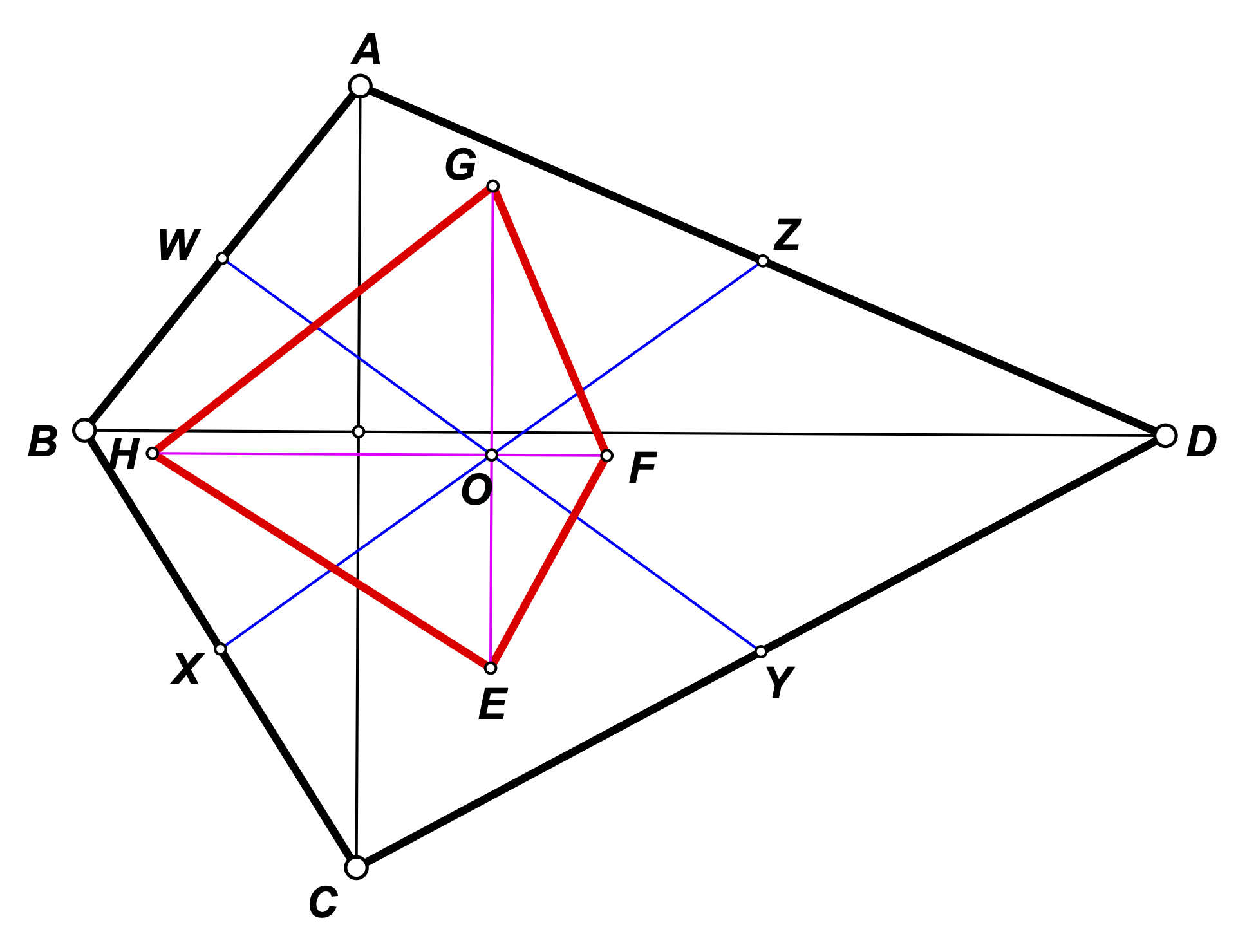}
\caption{orthodiagonal quad with $X_5$-points $\implies \m[ABCD]=\diag[EFGH]$}
\label{fig:odX5}
\end{figure}

Note that $W$, $X$, $Y$, and $Z$ are the midpoints of the sides of quadrilateral $ABCD$, making $O$ the centroid.
We need to show that $O$ coincides with the intersection of diagonals $EG$ and $FH$ of quadrilateral $EFGH$.

\begin{proof}
We begin by specifying the coordinates for the vertices of quadrilateral $ABCD$.

\begin{code}{0.4in}{5in}
\begin{verbatim}
ptA = {1:0:0};
ptB = {0:1:0};
ptC = {0:0:1};
ptD = {p:q:r};
\end{verbatim}
\end{code}

Then we use the function \texttt{CentroETCTriangulo} from the \texttt{baricentricas} package to create a routine that determines
the center $X_n$ of the four half triangles of quadrilateral $ABCD$.

\begin{code}{0.3in}{5in}
\begin{verbatim}
CentralQuadrilateral[n_] :=
   {
   Simplificar[CentroETCTriangulo[ETC[[n, 2]], {ptB, ptC, ptD}]],
   Simplificar[CentroETCTriangulo[ETC[[n, 2]], {ptA, ptC, ptD}]],
   Simplificar[CentroETCTriangulo[ETC[[n, 2]], {ptA, ptB, ptD}]],
   Simplificar[CentroETCTriangulo[ETC[[n, 2]], {ptA, ptB, ptC}]]
   };
\end{verbatim}
\end{code}

Then we use this routine to find the coordinates of $E$, $F$, $G$, and $H$.

\begin{code}{0.4in}{5in}
\begin{verbatim}
{ptE, ptF, ptG, ptH} = CentralQuadrilateral[5];
\end{verbatim}
\end{code}

\goodbreak
The result from Mathematica shows that
   $\mathrm{ptE}=\{x,y,z\}$ where
\begin{align*}
	x=&\left((b^2-c^2)^2-a^2( b^2+c^2)\right)p^2-2 a^4 q r-a^2\left(a^2+b^2-c^2\right)pr-a^2\left(a^2-b^2+c^2\right)pq,\\
	y=&\left((a^2-c^2)^2-b^2 (c^2+a^2)\right)p^2+\left(2 a^4+b^4+c^4-2 b^2 c^2-3 a^2 c^2-3 a^2 b^2\right)p q\\
	  &+\left(a^4+b^4+c^4-2 b^2 c^2-2 a^2 c^2\right)p r+a^2\left(a^2+b^2-c^2\right)qr,\\
	z=&\left((a^2-b^2)^2-c^2(a^2+b^2)\right)p^2+\left(a^4+b^4+c^4-2 b^2 c^2-2 a^2 b^2\right)pq\\
		&+\left(2 a^4+b^4+c^4-2 b^2 c^2-3 a^2 c^2-3 a^2 b^2\right)pr+a^2 q r \left(a^2-b^2+c^2\right)
\end{align*}

with similarly complicated expressions for ptF, ptG, and ptH.

Next, we find the centroid of quadrilateral $ABCD$.

\begin{code}{0.4in}{5in}
\begin{verbatim}
centroid = CentroidQuad[{ptA, ptB, ptC, ptD}];
\end{verbatim}
\end{code}

using the routine

\begin{code}{0.4in}{5in}
\begin{verbatim}
CentroidQuad[{P_, Q_, R_, S_}] := Punto[
   Recta[Medio[P, Q], Medio[R, S]],
   Recta[Medio[P, S], Medio[Q, R]]
   ];
\end{verbatim}
\end{code}

giving the result
$$\mathrm{centroid}=\{2 p + q + r, p + 2 q + r, p + q + 2 r\}.$$
Then we find the diagonal point of quadrilateral $EFGH$,

\begin{code}{0.4in}{5in}
\begin{verbatim}
dp = DiagonalPt[{ptE, ptF, ptG, ptH}];
\end{verbatim}
\end{code}

using the routine

\begin{code}{0.4in}{5in}
\begin{verbatim}
DiagonalPt[{P_, Q_, R_, S_}] := Punto[Recta[P, R], Recta[Q, S]];
\end{verbatim}
\end{code}

giving a complicated expression for dp.

Now we write down the condition that these two points coincide recalling the fact that
two barycentric coordinates represent the same point if they are proportional.

\begin{code}{0.4in}{5in}
\begin{verbatim}
sameCondition = Cross[centroid, dp] == {0, 0, 0};
\end{verbatim}
\end{code}

The resulting expression is quite complicated since the two points do not coincide in an arbitrary quadrilateral.

The next step is to find the condition that ensures that $ABCD$ is orthodiagonal.

\begin{code}{0.4in}{5in}
\begin{verbatim}
orthodiag = SonPerpendiculares[Recta[ptA, ptC], Recta[ptB, ptD]];
\end{verbatim}
\end{code}

The condition found is
$$a^2 (p+r)+b^2 (r-p)-c^2 (p+r)=0.$$
Finally, we simplify sameCondition subject to this constraint.

\begin{code}{0.4in}{5in}
\begin{verbatim}
Simplify[sameCondition, orthodiag]
\end{verbatim}
\end{code}

Mathematica responds with

\begin{code}{0.4in}{5in}
\begin{verbatim}
True
\end{verbatim}
\end{code}

indicating that the points coincide.
\end{proof}

\newpage

\section{General Quadrilaterals}

Our computer study found the following relationships between a general quadrilateral and its central quadrilateral.

\begin{center}
\begin{tabular}{|l|p{2.2in}|}
\hline
\multicolumn{2}{|c|}{\textbf{\color{blue}\large \strut Central Quadrilaterals of General Quadrilaterals}}\\ \hline
\textbf{Relationship}&\textbf{centers}\\ \hline
\ru $[ABCD]=9[EFGH]$&2\\ \hline
\ru $[ABCD]=[EFGH]$&4\\ \hline
\ru QA-P2=$\hyperb(ABCD,EFGH)$=QA-P2&4\\ \hline
\ru $\m[ABCD]=\m[EFGH]$&2\\ \hline
\ru $\m[ABCD]=\ponce[EFGH]$&5\\ \hline
\ru $\ponce[ABCD]=\ponce[EFGH]$&4\\ \hline
\ru $\ponce[EFGH]=\stein[ABCD]$&3\\ \hline
\ru QA-P1=$\ho[ABCD,EFGH]$=QA-P1&2\\ \hline
\end{tabular}
\end{center}


\subsection{Properties involving $X_2$}\ \\

The following result comes from \cite{shapes}.

\begin{theorem}
\label{thm:gqX2sim}
Let $ABCD$ be an arbitrary quadrilateral.
Let $E$, $F$, $G$, and $H$ be the $X_{2}$-points of $\triangle BCD$, $\triangle CDA$, 
$\triangle DAB$, and $\triangle ABC$, respectively.
Then quadrilaterals $ABCD$ and $EFGH$ are similar. The ratio of similitude is 3 (Figure~\ref{fig:gqX2sim}).
\end{theorem}

\begin{figure}[h!t]
\centering
\includegraphics[width=0.35\linewidth]{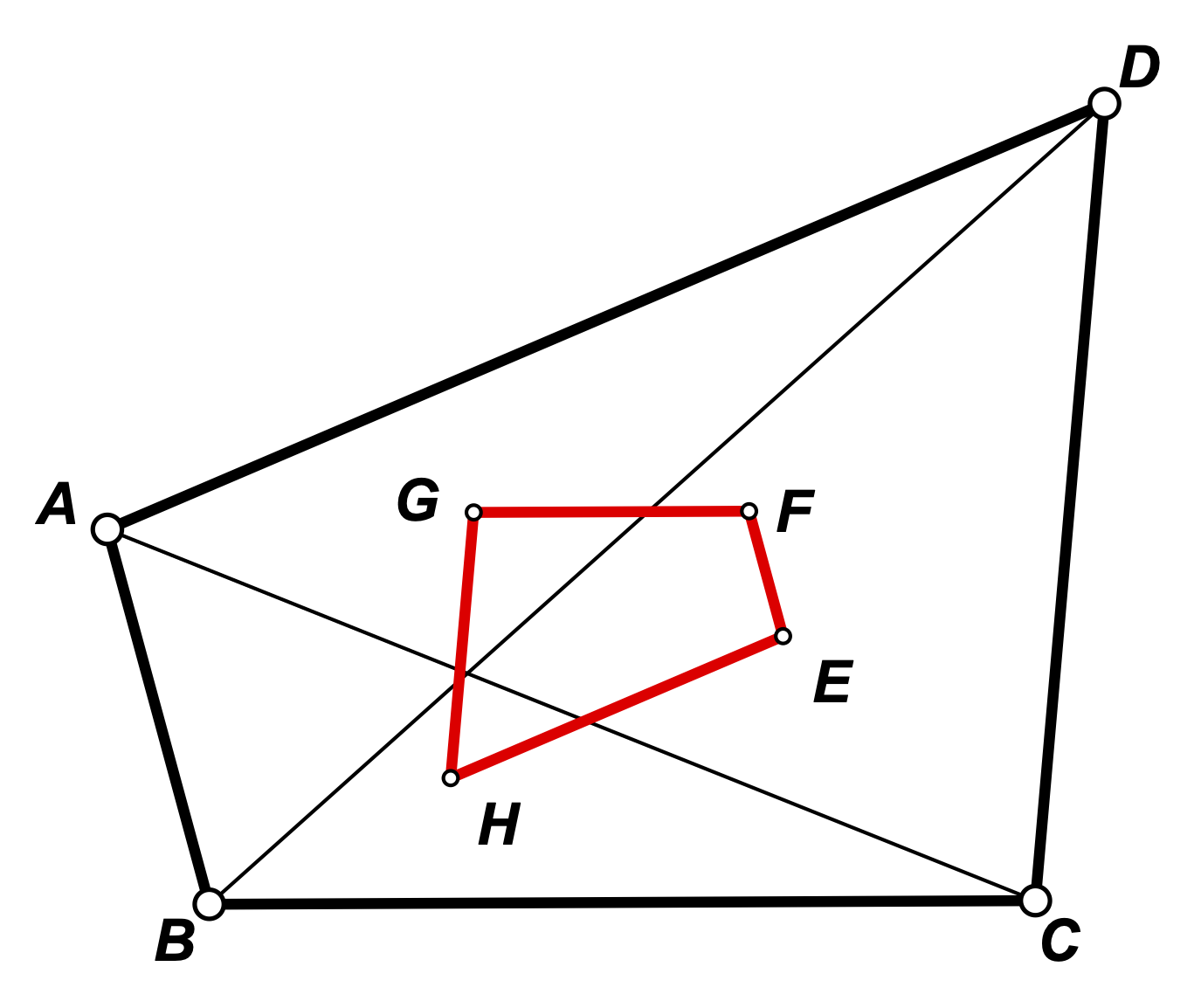}
\caption{general quadrilateral with $X_2$-points $\implies ABCD\sim EFGH$}
\label{fig:gqX2sim}
\end{figure}

\begin{theorem}
\label{thm:gqX2area}
Let $ABCD$ be an arbitrary quadrilateral.
Let $E$, $F$, $G$, and $H$ be the $X_{2}$-points of $\triangle BCD$, $\triangle CDA$, 
$\triangle DAB$, and $\triangle ABC$, respectively.
Then $[ABCD]=9[EFGH]$ (Figure~\ref{fig:gqX2sim}).
\end{theorem}

\begin{proof}
This follows immediately from Theorem~\ref{thm:gqX2sim} since the ratio of the areas of two similar
figures is equal to the square of the ratio of their sides.
\end{proof}

\newpage

\begin{theorem}
\label{thm:gqX2hom}
Let $ABCD$ be an arbitrary quadrilateral.
Let $E$, $F$, $G$, and $H$ be the $X_{2}$-points of $\triangle BCD$, $\triangle CDA$, 
$\triangle DAB$, and $\triangle ABC$, respectively.
Then quadrilaterals $ABCD$ and $EFGH$ are homothetic. The homothetic center is the centroid of quadrilateral $ABCD$ (Figure~\ref{fig:gqX2hom}).
\end{theorem}

\begin{figure}[h!t]
\centering
\includegraphics[width=0.35\linewidth]{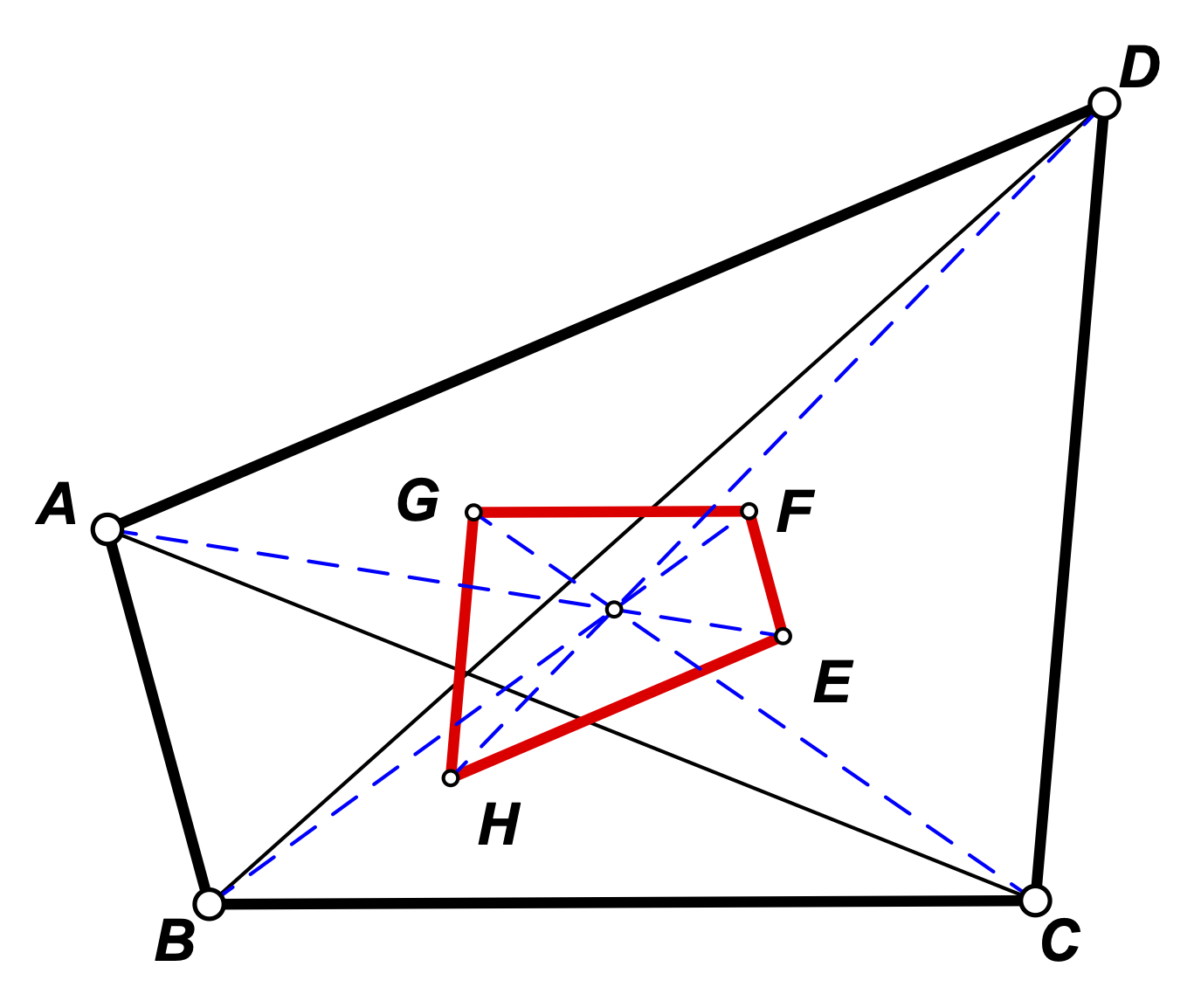}
\caption{general quadrilateral with $X_2$-points $\implies \ho(ABCD,EFGH)$}
\label{fig:gqX2hom}
\end{figure}

\begin{proof}
From \cite{QA-P1} we know that the lines from the vertices of a triangle to the centroid of the opposite half triangle
meet in a point known as the centroid of the quadrilateral. Thus, the quadrilaterals are perspective.
Since they are also similar, this means they are homothetic.
\end{proof}

\begin{theorem}
\label{thm:gqX2m}
Let $ABCD$ be an arbitrary quadrilateral.
Let $E$, $F$, $G$, and $H$ be the $X_{2}$-points of $\triangle BCD$, $\triangle CDA$, 
$\triangle DAB$, and $\triangle ABC$, respectively.
Then quadrilaterals $ABCD$ and $EFGH$ have the same centroid (Figure~\ref{fig:gqX2hom}).
\end{theorem}

\begin{proof}
This follows from Theorem~\ref{thm:gqX2hom} since a homothety maps the centroid of a figure into the centroid of the new figure
and the center of the homothety is the centroid of quadrilateral $ABCD$.
\end{proof}

\subsection{Properties involving $X_3$}\ \\

\begin{theorem}
\label{thm:gqX3}
Let $ABCD$ be an arbitrary quadrilateral.
Let $E$, $F$, $G$, and $H$ be the $X_{3}$-points of $\triangle BCD$, $\triangle CDA$, 
$\triangle DAB$, and $\triangle ABC$, respectively.
Then  $$\ponce[EFGH]=\stein[ABCD].$$
\end{theorem}

Our proof of Theorem~\ref{thm:gqX3} is analytical using barycentric coordinates.

\begin{open}
Is there a purely geometrical proof of Theorem~\ref{thm:gqX3}?
\end{open}

\subsection{Properties involving $X_4$}\ \\

The following result comes from \cite{shapes}.

\begin{theorem}
\label{thm:gqX4}
Let $ABCD$ be an arbitrary quadrilateral.
Let $E$, $F$, $G$, and $H$ be the $X_{4}$-points of $\triangle BCD$, $\triangle CDA$, 
$\triangle DAB$, and $\triangle ABC$, respectively.
Then quadrilaterals $ABCD$ and $EFGH$ have the same area (Figure~\ref{fig:gqX4}).
\end{theorem}

\begin{figure}[h!t]
\centering
\includegraphics[width=0.4\linewidth]{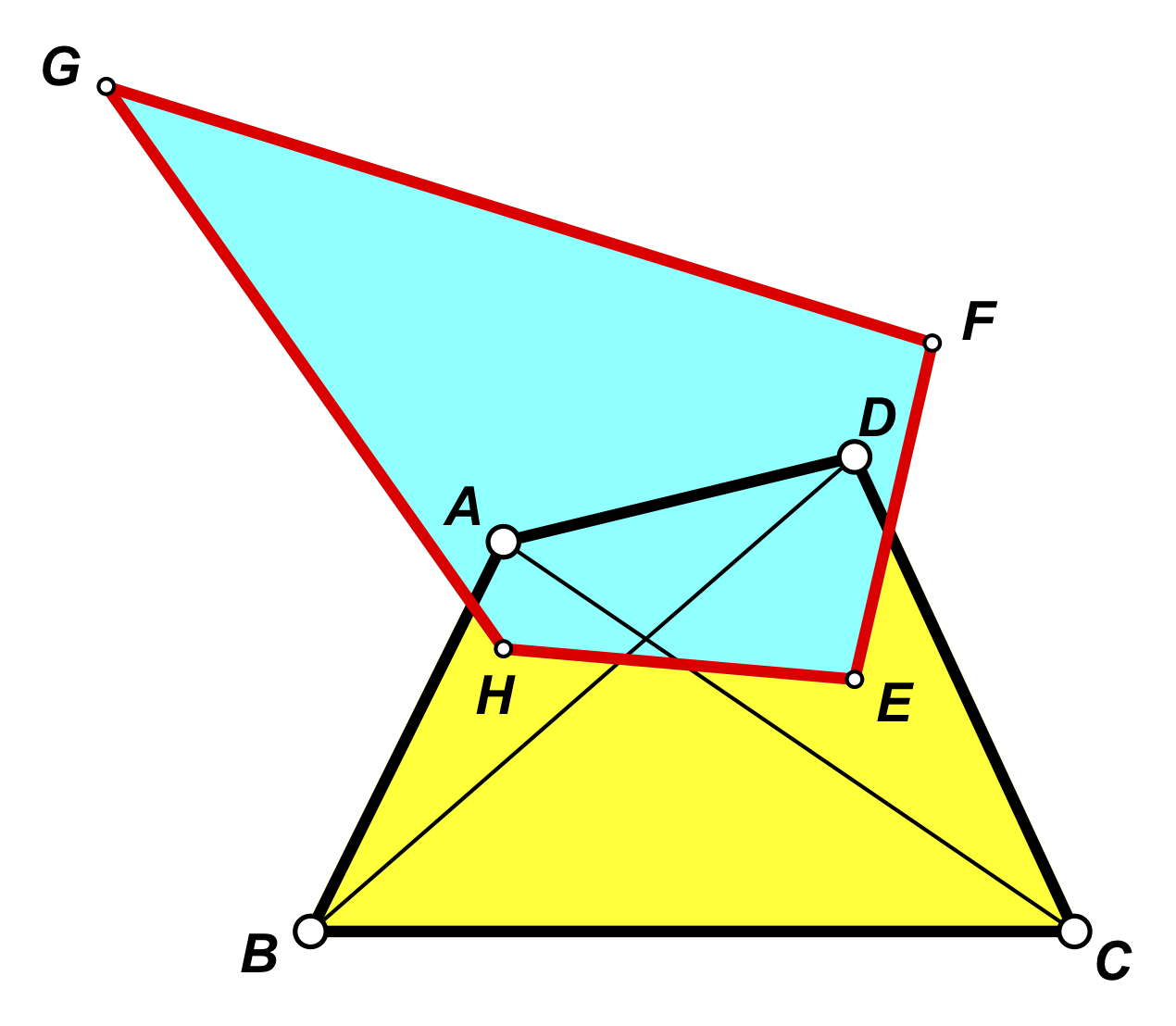}
\caption{general quadrilateral with $X_4$-points $\implies [ABCD]=[EFGH]$}
\label{fig:gqX4}
\end{figure}

\newpage

The following result comes from \cite{QA-P2}.

\begin{theorem}
\label{thm:gqX4conic}
Let $ABCD$ be an arbitrary quadrilateral.
Let $E$, $F$, $G$, and $H$ be the $X_{4}$-points of $\triangle BCD$, $\triangle CDA$, 
$\triangle DAB$, and $\triangle ABC$, respectively.
Then quadrilaterals $ABCD$ and $EFGH$ have a common circumconic (Figure~\ref{fig:gqX4conic}).
The conic is a rectangular hyperbola and
the center of the conic, $O$, is the Euler-Poncelet Point (QA-P2) of both quadrilaterals $ABCD$ and $EFGH$.
\end{theorem}

\begin{figure}[h!t]
\centering
\includegraphics[width=0.4\linewidth]{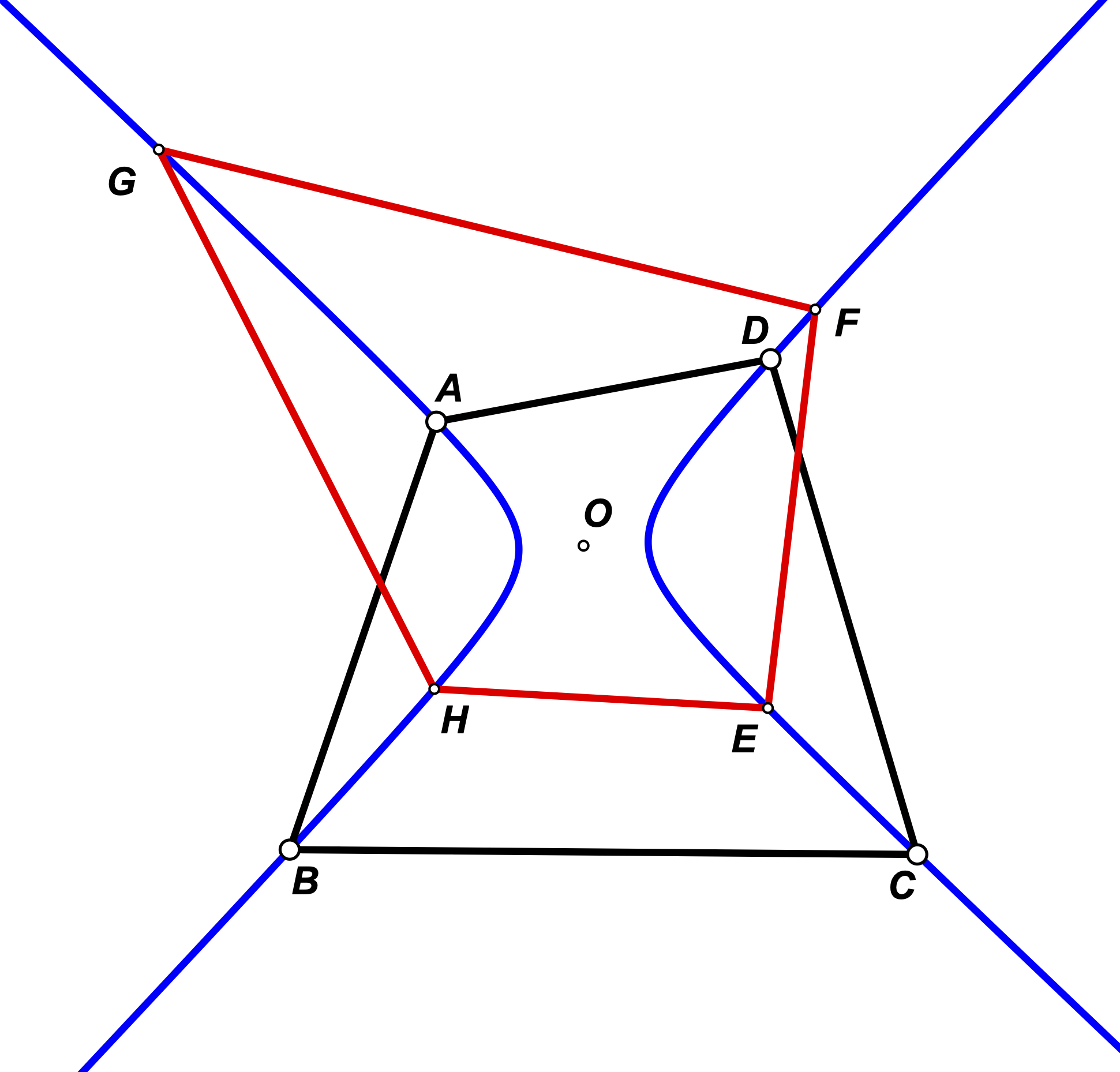}
\caption{general quadrilateral with $X_4$-points $\implies \hyperb(ABCD,EFGH)$}
\label{fig:gqX4conic}
\end{figure}

\begin{corollary}
Let $ABCD$ be an arbitrary quadrilateral.
Let $E$, $F$, $G$, and $H$ be the $X_{4}$-points of $\triangle BCD$, $\triangle CDA$, 
$\triangle DAB$, and $\triangle ABC$, respectively.
Then $$\ponce[ABCD]=\ponce[EFGH].$$
\end{corollary}

\subsection{Properties involving $X_5$}\ \\

\begin{theorem}
\label{thm:gqX5}
Let $ABCD$ be an arbitrary quadrilateral.
Let $E$, $F$, $G$, and $H$ be the $X_{5}$-points of $\triangle BCD$, $\triangle CDA$, 
$\triangle DAB$, and $\triangle ABC$, respectively.
Then  $$\m[EFGH]=\ponce[ABCD].$$
\end{theorem}

Our proof of Theorem~\ref{thm:gqX5} is analytical using barycentric coordinates.

\begin{open}
Is there a purely geometrical proof of Theorem~\ref{thm:gqX5}?
\end{open}

\newpage

\section{Tangential Quadrilaterals}

A \emph{tangential quadrilateral} in one in which a circle can be inscribed, touching all
four sides.
The center of this circle is called the \emph{incenter} of the quadrilateral.
The circle is called the \emph{incircle}.

Our computer study found only one relationship between a tangential quadrilateral
and its central quadrilateral (using any of the first 1000 centers) that was not true for quadrilaterals in general.
It is listed in the following table.
\medskip

\begin{center}
\begin{tabular}{|l|p{2.2in}|}
\hline
\multicolumn{2}{|c|}{\textbf{\color{blue}\large \strut Central Quadrilaterals of Tangential Quadrilaterals}}\\ \hline
\textbf{Relationship}&\textbf{centers}\\ \hline
\ru $\incent[ABCD]=\persp[ABCD,GHEF]$&1\\ \hline
\end{tabular}
\end{center}

\bigskip

\begin{theorem}
\label{thm:tqX1}
Let $ABCD$ be a tangential quadrilateral with incenter $I$.
Let $E$, $F$, $G$, and $H$ be the $X_{1}$-points of $\triangle BCD$, $\triangle CDA$, 
$\triangle DAB$, and $\triangle ABC$, respectively.
Then quadrilaterals $ABCD$ and $GHEF$ are perspective with perspector $I$ (Figure~\ref{fig:tqX1}).
\end{theorem}

\begin{figure}[h!t]
\centering
\includegraphics[width=0.4\linewidth]{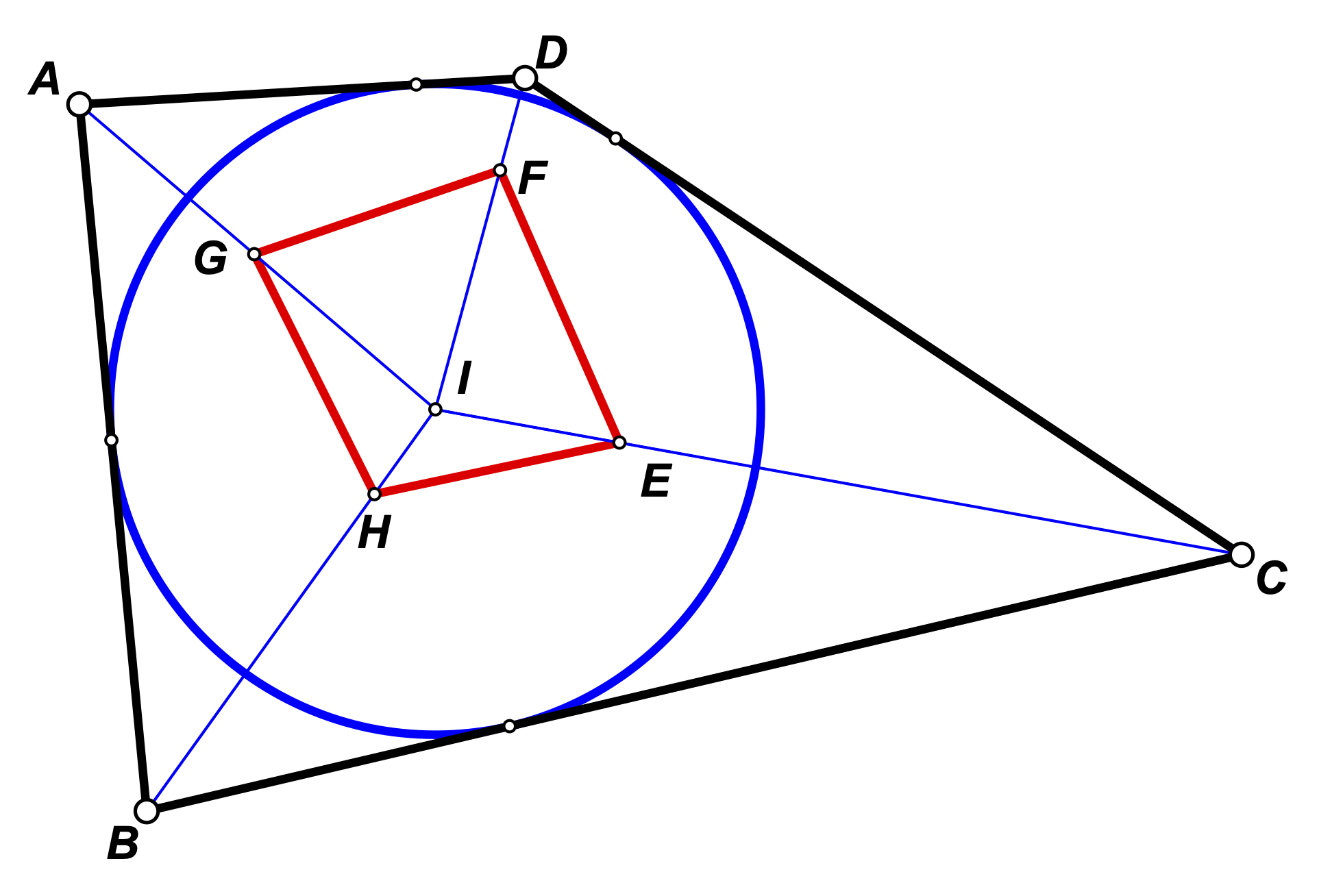}
\caption{tangential quadrilateral with $X_1$-points $\implies \persp[ABCD,GHEF]$}
\label{fig:tqX1}
\end{figure}

\begin{proof}
The point $G$ is the incenter of $\triangle ABD$, hence $G$ lies on the angle bisector of $\angle BAD$. Thus, $G\in AI$.
Similarly, $H\in BI$, $E\in CI$, and $F\in DI$. Therefore, $AG$, $BH$, $CE$, and $DF$ concur in $I$.
Hence, quadrilaterals $ABCD$ and $GHEF$ are perspective and the perspector is $I$.
\end{proof}

\begin{open}
For the quadrilaterals in Theorem~\ref{thm:tqX1}, how is the perspector related to quadrilateral $EFGH$?
\end{open}

\section{Extangential Quadrilaterals}

An \emph{extangential quadrilateral} with consecutive sides of lengths $a$, $b$, $c$, and $d$ is one in which $a+b=c+d$.

Our computer study did not find any relationships between an extangential quadrilateral
and its central quadrilateral (using any of the first 1000 centers) that was not true for quadrilaterals in general.
\medskip

\begin{center}
\begin{tabular}{|l|p{2.2in}|}
\hline
\multicolumn{2}{|c|}{\textbf{\color{blue}\large \strut Central Quadrilaterals of exTangential Quadrilaterals}}\\ \hline
\multicolumn{2}{|c|}{No new relationships were found.}\\ \hline
\end{tabular}
\end{center}

\section{EqualProdOp Quadrilaterals}

An \emph{equalProdOp quadrilateral} with consecutive sides of lengths $a$, $b$, $c$, and $d$ is one in which $ac=bd$.

Our computer study did not find any relationships between an equalProdOp quadrilateral
and its central quadrilateral (using any of the first 1000 centers) that were not true for quadrilaterals in general.
\medskip

\begin{center}
\begin{tabular}{|l|l|p{2.2in}|}
\hline
\multicolumn{3}{|c|}{\textbf{\color{blue}\large \strut Central Quadrilaterals of EqualProdOp Quadrilaterals}}\\ \hline
\multicolumn{3}{|c|}{No new relationships were found.}\\
\hline
\end{tabular}
\end{center}

\section{EqualProdAdj Quadrilaterals}

An \emph{equalProdAdj quadrilateral} with consecutive sides of lengths $a$, $b$, $c$, and $d$ is one in which $ab=cd$.

Our computer study did not find any relationships between an equalProdAdj quadrilateral
and its central quadrilateral (using any of the first 1000 centers) that were not true for quadrilaterals in general.
\medskip

\begin{center}
\begin{tabular}{|l|l|p{2.2in}|}
\hline
\multicolumn{3}{|c|}{\textbf{\color{blue}\large \strut Central Quadrilaterals of EqualProdAdj Quadrilaterals}}\\ \hline
\multicolumn{3}{|c|}{No new relationships were found.}\\
\hline
\end{tabular}
\end{center}

\section{Pythagorean Quadrilaterals}

A \emph{Pythagorean quadrilateral} with consecutive sides of lengths $a$, $b$, $c$, and $d$ is one in which $a^2+b^2=c^2+d^2$.

Our computer study did not find any relationships between a Pythagorean  quadrilateral
and its central quadrilateral (using any of the first 1000 centers) that were not true for quadrilaterals in general.
\medskip

\begin{center}
\begin{tabular}{|l|l|p{2.2in}|}
\hline
\multicolumn{3}{|c|}{\textbf{\color{blue}\large \strut Central Quadrilaterals of Pythagorean Quadrilaterals}}\\ \hline
\multicolumn{3}{|c|}{No new relationships were found.}\\
\hline
\end{tabular}
\end{center}

\section{Equidiagonal Quadrilaterals}

An \emph{equidiagonal quadrilateral} is a quadrilateral with two equal diagonals.

Our computer study did not find any relationships between an equidiagonal quadrilateral
and its central quadrilateral (using any of the first 1000 centers) that were not true for quadrilaterals in general.
\medskip

\begin{center}
\begin{tabular}{|l|l|p{2.2in}|}
\hline
\multicolumn{3}{|c|}{\textbf{\color{blue}\large \strut Central Quadrilaterals of Equidiagonal Quadrilaterals}}\\ \hline
\multicolumn{3}{|c|}{No new relationships were found.}\\
\hline
\end{tabular}
\end{center}

\newpage

\section{Orthodiagonal Quadrilaterals}

An \emph{orthodiagonal quadrilateral} is a quadrilateral in which the two diagonals are perpendicular.

Our computer study found a few relationships between an orthodiagonal quadrilateral
and its central quadrilateral (using any of the first 1000 centers) that are not true for quadrilaterals in general.
These are shown in the following table.
\medskip

\begin{center}
\begin{tabular}{|l|p{2.2in}|}
\hline
\multicolumn{2}{|c|}{\textbf{\color{blue}\large \strut Central Quadrilaterals of Orthodiagonal Quadrilaterals}}\\ \hline
\textbf{Relationship}&\textbf{centers}\\ \hline
\ru $\stein[ABCD]=\diag(EFGH)$&3\\ \hline
\ru QA-P4=$\persp[ABCD,GHEF]$=QA-P4&3\\ \hline
\ru $\diag(ABCD)=\diag(EFGH)=\ponce[ABCD]$&4\\ \hline
\ru $\m[ABCD]=\diag(EFGH)$&5\\ \hline
\ru $\persp[ABCD,GHEF]$&25, 68, 485, 486\\ \hline
\end{tabular}
\end{center}

\bigskip

\begin{theorem}
\label{thm:odX4}
Let $ABCD$ be an orthodiagonal quadrilateral.
Let $E$, $F$, $G$, and $H$ be the $X_{4}$ -points of $\triangle BCD$, $\triangle CDA$, 
$\triangle DAB$, and $\triangle ABC$, respectively.
Then quadrilaterals $ABCD$ and $EFGH$ have the same diagonal point
and $\diag(ABCD)=\ponce[ABCD]$ (Figure~\ref{fig:odX4}).
\end{theorem}

\begin{figure}[h!t]
\centering
\includegraphics[width=0.4\linewidth]{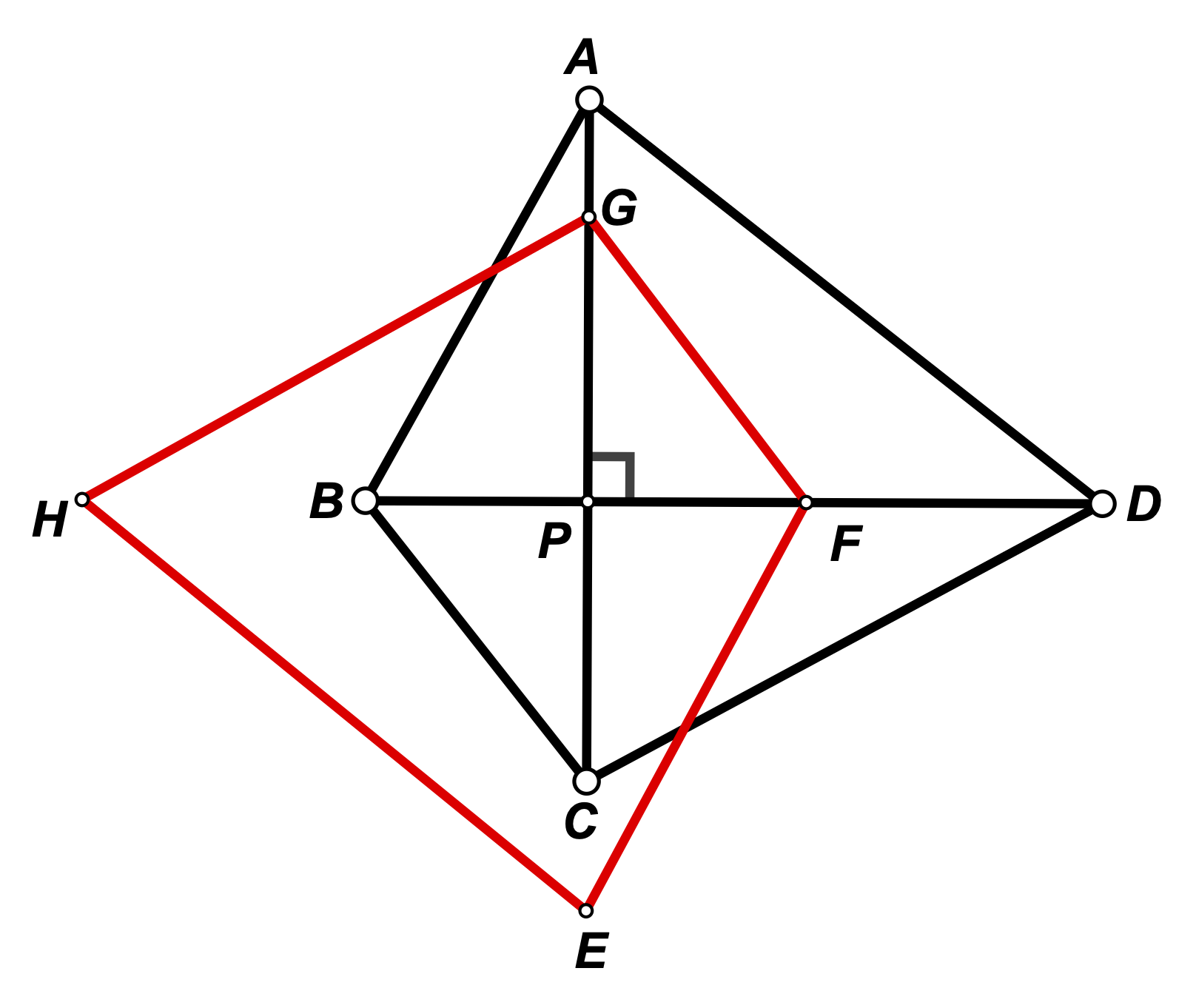}
\caption{\small orthodiagonal quad with $X_4$-points $\implies \diag(ABCD)=\diag(EFGH)$}
\label{fig:odX4}
\end{figure}

\begin{proof}
Let $P$ be the diagonal point of quadrilateral $ABCD$.
Since $ABCD$ is orthodiagonal, $AP$ is an altitude of $\triangle ABD$.
Since $E$ is the orthocenter of $\triangle BCD$, $E$ lies on this altitude.
Similarly, $G$ also lies on this altitude.
In the same way, points $F$ and $H$ lie on $BD$.
Thus, the diagonals $EG$ and $FH$ of central quadrilateral $EFGH$ meet at $P$ and so $EFGH$ has diagonal point $P$.
Hence, $ \diag(ABCD)=\diag(EFGH)$.

The point $P$ is the foot of the $C$-altitude of $\triangle BCD$, so $C$ lies on the ninepoint circle of $\triangle BCD$.
Similarly $P$ lies on the ninepoint circles of the other half triangles. Therefore, $P=\ponce[ABCD]$ and $\diag(ABCD)=\ponce[ABCD]$.
\end{proof}

\begin{open}
Are there purely geometrical proofs for the results found in this section involving centers $X_3$, $X_4$, and $X_5$?
\end{open}

\newpage

\section{Cyclic Quadrilaterals}
\label{section:cyclic}

\newcommand{\cyc}[1]{\textcolor{blue}{#1}}

An \emph{cyclic quadrilateral} is a quadrilateral that can be inscribed in a circle.

Our computer study found many relationships between a cyclic quadrilateral
and its central quadrilateral. They are summarized in the following tables.
Properties that are true for quadrilaterals in general are excluded.

Centers that are colored \cyc{blue} in the following table represent centers for which the central quadrilateral is cyclic.

\bigskip
\bigskip

\begin{center}
\begin{tabular}{|p{3in}|p{1in}|}
\hline
\multicolumn{2}{|c|}{\textbf{\color{blue}\large \strut Central Quadrilaterals of Cyclic Quadrilaterals}}\\ \hline
\textbf{Relationship}&\textbf{centers}\\ \hline

\ru $[ABCD]=4[EFGH]$& \cyc{5, 550}\\ \hline
\ru $[ABCD]=9[EFGH]$& \cyc{376}\\ \hline
\ru $[ABCD]=16[EFGH]$& \cyc{140, 548}\\ \hline
\ru $[ABCD]=25[EFGH]$& \cyc{631}\\ \hline
\ru $[ABCD]=36[EFGH]$& \cyc{549}\\ \hline
\ru $[ABCD]/[EFGH]=1/4$& \cyc{382}\\ \hline
\ru $[ABCD]/[EFGH]=9/4$& \cyc{381}\\ \hline
\ru $[ABCD]/[EFGH]=16/9$& \cyc{546}\\ \hline
\ru $[ABCD]/[EFGH]=100/9$& \cyc{632}\\ \hline
\ru $[ABCD]/[EFGH]=144/25$& \cyc{547}\\ \hline
\ru $\conic(ABCD,EFGH)$& 6, 54, 64\\ \hline
\ru $ABCD\cong EFGH$& \cyc{4}, \cyc{20}\\ \hline
\ru $\anti[ABCD]=\anti[EFGH]$& \cyc{4}\\ \hline
\ru $\anti[ABCD]=\centro[EFGH]$& \cyc{546}\\ \hline
\ru $\anti[ABCD]=\m[EFGH]$& \cyc{381}\\ \hline
\ru $\anti[ABCD]=\oo[EFGH]$& \cyc{5}\\ \hline
\ru $\anti[EFGH]=\centro[ABCD]$& \cyc{381}\\ \hline
\ru $\anti[EFGH]=\h[ABCD]$& \cyc{382}\\ \hline
\ru $\anti[EFGH]=\m[ABCD]$& \cyc{5}\\ \hline
\ru $\centro[ABCD]=\centro[EFGH]$& \cyc{5}\\ \hline
\ru $\centro[ABCD]=\oo[EFGH]$& \cyc{2}\\ \hline
\ru $\h[ABCD]=\oo[EFGH]$& \cyc{4}\\ \hline
\ru $\m[ABCD]=\oo[EFGH]$& \cyc{140}\\ \hline
\end{tabular}
\end{center}

\newpage

\begin{center}
\begin{tabular}{|p{3.1in}|p{2.6in}|}
\hline
\multicolumn{2}{|c|}{\textbf{\color{blue}\large \strut Central Quadrilaterals of Cyclic Quadrilaterals (cont.)}}\\ \hline
\textbf{Relationship}&\textbf{centers}\\ \hline

\ru \raggedright$\oo[ABCD]=\h[EFGH]$& \cyc{$\mathbb{S}$}\\ \hline
\ru $\diag(ABCD)=\diag(EFGH)$& 6\\ \hline
\ru $\diag(ABDC)=\diag(EFHG)$& 6, \cyc{15, 16}, 61, 62, 371, 372\\ \hline
\ru $\diag(ACBD)=\diag(EGFH)$& 6, \cyc{15, 16}, 61, 62, 371, 372\\ \hline
\ru $(ABCD)\equiv(EFGH)$&$\mathbb{C}$\\ \hline
\ru \raggedright$\oo[ABCD]=\oo[EFGH]$&399\\ \hline
\ru QA-P2=$\ho[ABCD,EFGH]$=QA-P2& \cyc{4}\\ \hline
\ru QA-P7=$\ho[ABCD,EFGH]$=QA-P7& \cyc{5}\\ \hline
\ru QA-P34=$\ho[ABCD,EFGH]$=QA-P34& \cyc{631}\\ \hline
\ru $\ho[ABCD,EFGH]$& \cyc{$\mathbb{S}$}\\ \hline
\end{tabular}
\end{center}

The symbol $\mathbb{C}$ denotes the set of all triangle centers that lie on the circumcircle of the reference triangle.

The list of triangle centers that lie on the circumcircle of the reference triangle
can be found in \cite{MathWorld-Circumcircle}.
The first few are $X_n$ for $n=$74, 98--112, 476, 477, 675, 681, 689, 691, 697, 699, 701, 703, 705, 707, 709, 711, 713, 715, 717, 719, 721, 723, 725, 727, 729, 731, 733, 735, 737, 739, 741, 743, 745, 747, 753, 755, 759, 761, 767, 769, 773, 777, 779, 781, 783, 785, 787, 789, 791, 793, 795, 797, 803, 805, 807, 809, 813, 815, 817, 819, 825, 827, 831, 833, 835, 839--843, 898, 901, 907, 915, 917, 919, 925, 927, 929--935, 953, and 972.

The triangle centers that do not lie on the circumcircle, for which the central quadrilateral of a cyclic quadrilateral is cyclic
are $X_n$ for
$n=$1, 2, 4, 5, 13--16, 20, 23, 36, 40, 80, 125, 140, 165, 186, 265, 376, 381, 382, 399, 546--550, 631, 632.
The only one where the circumcircle of the central quadrilateral is concentric with the circumcircle of the reference triangle is $X_{399}$.

\medskip
The symbol $\mathbb{S}$ denotes the set of all triangle centers that lie on the Euler line of the reference triangle and have constant Shinagawa coefficients.
Shinagawa coefficients are defined in \cite{ETC}.
The first few $n$ for which $X_n$ has constant Shinagawa coefficients are $n=$2, 3, 4, 5, 20, 140, 376, 381, 382, 546-550, 631, and 632.

\medskip
The following are known facts about cyclic quadrilaterals.

\begin{theorem}
The Gergonne-Steiner point (QA-P3) of a cyclic quadrilateral coincides with the circumcenter of that quadrilateral.
That is, $$\stein[ABCD]=\oo[ABCD].$$
\end{theorem}

Relationships of this form will be excluded from our tables.
\goodbreak

\begin{theorem}
The Euler-Poncelet point (QA-P2) of a cyclic quadrilateral coincides with the anticenter of that quadrilateral.
That is, $$\ponce[ABCD]=\anti[ABCD].$$
\end{theorem}

Relationships of this form will be excluded from our tables.

\begin{theorem}
The Quadrangle Nine-point Homothetic Center (QA-P7) of a cyclic quadrilateral coincides with the centrocenter of that quadrilateral.
That is, 
$$\hbox{QA-P7}\;[ABCD]=\centro[ABCD].$$
\end{theorem}

\begin{theorem}
\label{thm:cqX399}
Let $ABCD$ be a cyclic quadrilateral.
Let $E$, $F$, $G$, and $H$ be the $X_{399}$-points of $\triangle BCD$, $\triangle CDA$, 
$\triangle DAB$, and $\triangle ABC$, respectively.
Then quadrilaterals $ABCD$ and $EFGH$ have concentric circumcircles
and their radii are in the ratio $1:2$(Figure~\ref{fig:cqX399}).
\end{theorem}

\begin{figure}[h!t]
\centering
\includegraphics[width=0.4\linewidth]{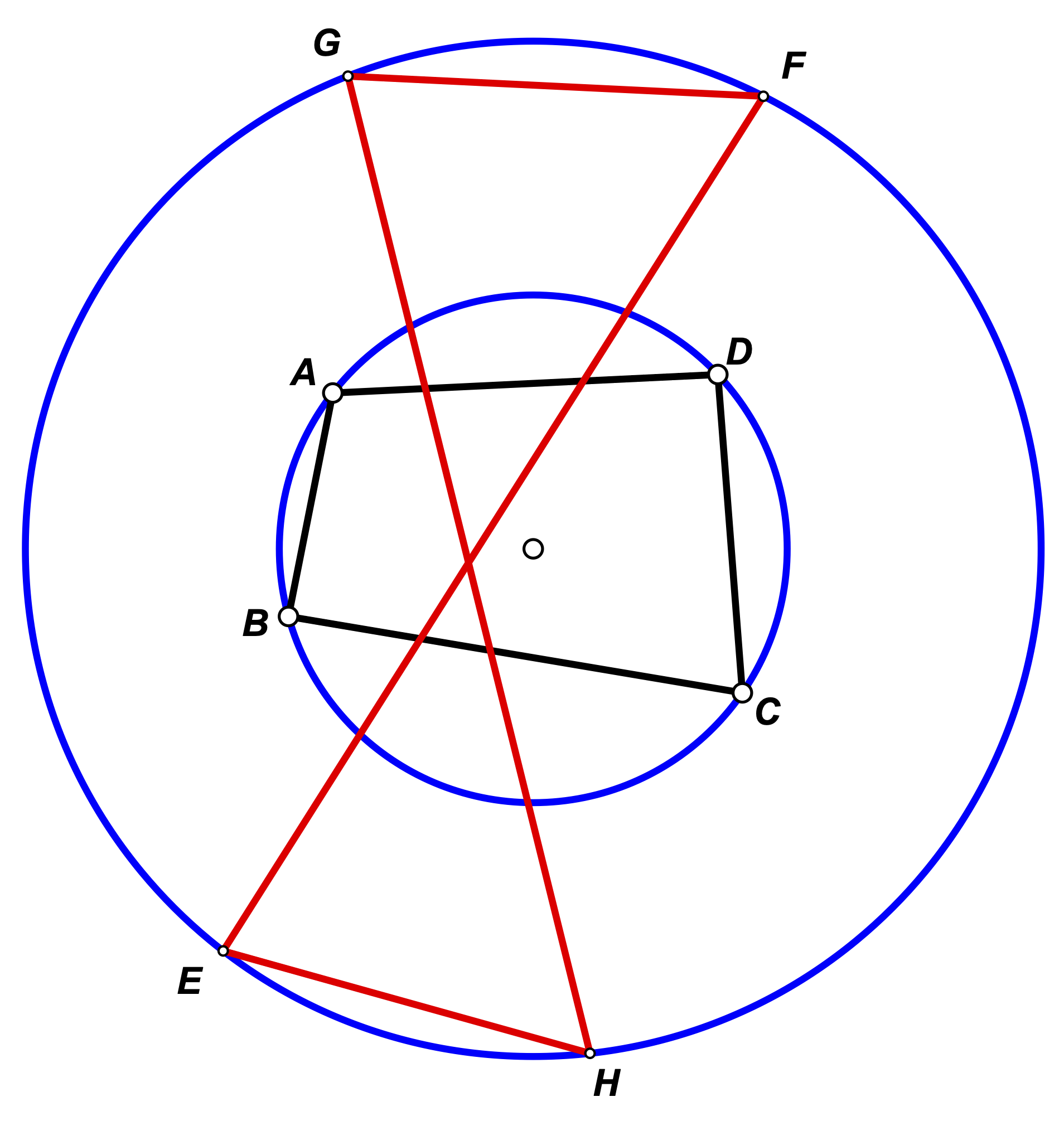}
\caption{cyclic quad with $X_{399}$-points $\implies $  concentric$[1:2](ABCD,EFGH)$}
\label{fig:cqX399}
\end{figure}

An analytic proof of Theorem~\ref{thm:cqX399}
is given in the supplementary material accompanying the on-line publication of this paper.
 
\begin{theorem}
\label{thm:cqXC}
Let $ABCD$ be a cyclic quadrilateral.
Let $X$ be any triangle center that lies on the circumcircle of the reference triangle.
Let $E$, $F$, $G$, $H$ be the $X$-points of $\triangle BCD$, $\triangle CDA$, 
$\triangle DAB$, and $\triangle ABC$, respectively.
Then quadrilaterals $ABCD$ and $EFGH$ have the same circumcircle, i.e. $(ABCD)\equiv(EFGH)$.
\end{theorem}

\begin{proof}
Let $\Gamma$ be the circumcircle of quadrilateral $ABCD$.
Since $E$ is the $X$-point of $\triangle BCD$, $E$ must lie on the circumcircle of $\triangle BCD$.
Hence $E$ lies on $\Gamma$. In the same way, $F$, $G$, and $H$ must also lie on $\Gamma$.
Therefore. $ABCD$ and $EFGH$ have the same circumcircle, $\Gamma$.
\end{proof}

\begin{open}
Are there purely geometrical proofs for the results found in this section involving centers $X_2$, $X_3$, $X_4$, $X_5$, and $X_6$?
\end{open}

\section{Bicentric Quadrilaterals}

A \emph{bicentric quadrilateral} is a quadrilateral that is both cyclic and tangential.

Our computer study found a few relationships between a bicentric quadrilateral
and its central quadrilateral (using any of the first 1000 centers) that were not true for cyclic quadrilaterals in general.
\medskip

\begin{center}
\begin{tabular}{|l|l|p{2.2in}|}
\hline
\multicolumn{2}{|c|}{\textbf{\color{blue}\large \strut Central Quadrilaterals of Bicentric Quadrilaterals}}\\ \hline
\textbf{Relationship}&\textbf{centers}\\ \hline
\ru $\persp[ABCD,GHEF]$& \no{35, 36, 55, 56, 999}\\ \hline
\end{tabular}
\end{center}


\bigskip

\void{

\begin{theorem}
\label{thm:bqX1}
Let $ABCD$ be a bicentric quadrilateral.
Let $E$, $F$, $G$, and $H$ be the $X_{1}$-points of $\triangle BCD$, $\triangle CDA$, 
$\triangle DAB$, and $\triangle ABC$, respectively (Figure~\ref{fig:bqX1}).
Then quadrilaterals $ABCD$ and $GHEF$ are perspective with perspector $I$.
Also, quadrilateral $EFGH$ is a rectangle.
\end{theorem}

\begin{figure}[h!t]
\centering
\includegraphics[width=0.35\linewidth]{bqX1.png}
\caption{bicentric quadrilateral with $X_1$-points $\implies \persp[ABCD,GHEF]$}
\label{fig:bqX1}
\end{figure}
}

\void{
If $ABCD$ is a quadrangle, the \emph{isotomic conjugate quadrangle} of $ABCD$ is formed by taking
the isotomic conjugate of each vertex of $ABCD$ with respect to the other three vertices.
}

\void{
\begin{proof}
Quadrilaterals $ABCD$ and $GHEF$ are perspective with perspector $I$ by Theorem~\ref{thm:tqX1},
since a bicentric quadrilateral is tangential.
Quadrilateral $EFGH$ is a rectangle by Theorem  B.2 of \cite{shapes}.
\end{proof}
}


\section{Trapezoids}

An \emph{trapezoid} is a quadrilateral with a pair of parallel sides.

Our computer study found only one relationship between a trapezoid
and its central quadrilateral (using any of the first 1000 centers) that is not true for quadrilaterals in general.
It is shown in the following table.
\medskip

\begin{center}
\begin{tabular}{|l|p{2.2in}|}
\hline
\multicolumn{2}{|c|}{\textbf{\color{blue}\large \strut Central Quadrilaterals of Trapezoids}}\\ \hline
\textbf{Relationship}&\textbf{centers}\\ \hline
\ru $ABCD\sim HGFE$&3\\ \hline
\end{tabular}
\end{center}

\void{
\bigskip

\begin{theorem}
\label{thm:trX3}
Let $ABCD$ be a trapezoid.
Let $E$, $F$, $G$, and $H$ be the $X_{3}$ -points of $\triangle BCD$, $\triangle CDA$, 
$\triangle DAB$, and $\triangle ABC$, respectively.
Then quadrilaterals $ABCD$ and $HGFE$ are similar and orthogonal (Figure~\ref{fig:trX3proof}).
\end{theorem}

\begin{figure}[h!t]
\centering
\includegraphics[width=0.35\linewidth]{trX3proof}
\caption{$X_3$-points in half triangles of a trapezoid}
\label{fig:trX3proof}
\end{figure}

\begin{proof}
(???) Since $E$ is the circumcenter of $\triangle BCD$, $EB=EC$.
Thus $E$ lies on the perpendicular bisector of $BC$.
Similarly, $H$ lies on the perpendicular bisector of $BC$.
Hence $EH\perp BC$.
In the same way, $FG\perp AD$.
Since $AD\parallel BC$, we can conclude that $EH\parallel FG$.
In the same way, $HG\perp AB$, $EF\perp CD$, $FH\perp AC$, and $GE\perp AB$.

The angle between two lines is the same as the angle between the perpendiculars to these lines.
Thus $\angle DAB=\angle EHG$, $\angle ABC=\angle HGF$, $\angle CDA=\angle FEH$, $\angle BCD=\angle GFE$,
$\angle EHF=\angle DAC$, and $\angle GFH=\angle BCS$.
Thus $\triangle ABC\sim\triangle HGF$ and $\triangle CDA\sim \triangle FEH$.
Thus $ABCD\sim HGFE$.

This proof is fallacious because the angle between two lines is sometimes the supplement of the angle between the perpendiculars to these lines.
Also, it doesn't use the fact that $AD\parallel BC$ which is a necessary condition.
\end{proof}
}

\section{Tangential Trapezoids}

A \emph{tangential trapezoid} is a trapezoid that is also tangential.

Our computer study did not find any relationships between a tangential trapezoid
and its central quadrilateral (using any of the first 1000 centers) that were not true for trapezoids or tangential quadrilaterals in general.
\medskip

\begin{center}
\begin{tabular}{|l|l|p{2.2in}|}
\hline
\multicolumn{3}{|c|}{\textbf{\color{blue}\large \strut Central Quadrilaterals of Tangential Trapezoids}}\\ \hline
\multicolumn{3}{|c|}{No new relationships were found.}\\
\hline
\end{tabular}
\end{center}

\section{Orthodiagonal Trapezoids}

An \emph{orthodiagonal trapezoid} is a trapezoid that is also orthodiagonal.

Our computer study did not find any relationships between an orthodiagonal trapezoid
and its central quadrilateral (using any of the first 1000 centers) that were not true for trapezoids in general or for orthodiagonal quadrilaterals.
\medskip

\begin{center}
\begin{tabular}{|l|p{2.2in}|}
\hline
\multicolumn{2}{|c|}{\textbf{\color{blue}\large \strut Central Quadrilaterals of Orthodiagonal Trapezoids}}\\ \hline
\multicolumn{2}{|c|}{No new relationships were found.}\\
\hline
\end{tabular}
\end{center}

\void{
\bigskip
\textbf{Note.} If the orthodiagonal trapezoid is also isosceles, we found the unusual result that the central quadrilateral formed from the $X_{930}$
points of the half triangles of quadrilateral $ABCD$ has a common circumconic with $ABCD$.
}


\section{Hjelmslev Quadrilaterals}

A \emph{Hjelmslev quadrilateral} is a quadrilateral with two right angles at opposite vertices.
Hjelmslev quadrilaterals are necessarily cyclic.

Our computer study did not find any relationships between a Hjelmslev quadrilateral
and its central quadrilateral (using any of the first 1000 centers) that were not true for cyclic quadrilaterals in general.
\medskip

\begin{center}
\begin{tabular}{|l|p{2.2in}|}
\hline
\multicolumn{2}{|c|}{\textbf{\color{blue}\large \strut Central Quadrilaterals of Hjelmslev Quadrilaterals}}\\ \hline
\multicolumn{2}{|c|}{No new relationships were found.}\\
\hline
\end{tabular}
\end{center}

\newpage

\section{Isosceles Trapezoids}

An \emph{isosceles trapezoid} is a trapezoid with its nonparallel sides having the same length.
Isosceles trapezoids are necessarily cyclic.

Our computer study found a few relationships between an isosceles trapezoid
and its central quadrilateral (using any of the first 1000 centers) that were not true for cyclic quadrilaterals in general.
They are given in the table below.
\medskip

\begin{center}
\begin{tabular}{|l|p{2.7in}|}
\hline
\multicolumn{2}{|c|}{\textbf{\color{blue}\large \strut Central Quadrilaterals of Isosceles Trapezoids}}\\ \hline
\ru $\persp[ABCD,HGFE]$& 19, 25, 48, 49, 63, 69, 186, 264, 265, 304, 305, 317, 340, 847\\ \hline
\ru QA-P9=$\persp[ABCD,HGFE]$& 24\\ \hline
\end{tabular}
\end{center}



\section{Harmonic Quadrilaterals}
\label{section:harmonic}

A \emph{harmonic quadrilateral} is a cyclic quadrilateral that is also an equalProdOpp quadrilateral.

Our computer study found a few relationships between a harmonic quadrilateral
and its central quadrilateral (using any of the first 1000 centers) that were not true for cyclic quadrilaterals in general.
They are shown in the following table.
\medskip

\begin{center}
\begin{tabular}{|l|l|p{2.2in}|}
\hline
\multicolumn{2}{|c|}{\textbf{\color{blue}\large \strut Central Quadrilaterals of Harmonic Quadrilaterals}}\\ \hline
\textbf{Relationship}&\textbf{centers}\\ \hline
\ru $\persp[ABCD,EFGH]$& 13--18, 61, 62, 371, 372, 395--398, 485, 590, 615\\ \hline
\ru $\persp[ABCD,GHEF]$& 15, 16, 61, 62, 371, 372\\ \hline
\end{tabular}
\end{center}

\bigskip
When proving these results analytically using barycentric coordinates, we use the following result.

\begin{theorem}
Let $ABCD$ be a harmonic quadrilateral. The barycentric coordinates of $D$ with respect to $\triangle ABC$ are $\left(2 a^2: -b^2: 2 c^2\right)$,
where $a=BC$, $b=CA$, and $c=AB$.
\end{theorem}

\begin{proof}
It is known \cite{Wiki-Harmonic} that the point $D$ is the second intersection of the $B$-symmedian with the circumcircle of $\triangle ABC$.

\begin{figure}[h!t]
\centering
\includegraphics[scale=1]{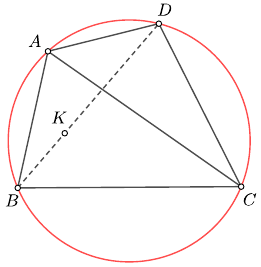}
\label{fig:harmonic}
\end{figure}

The coordinates of the symmedian point, $K$, are $\left(a^2:b^2:c^2\right)$, so the equation of the $B$-symmedian is
\begin{equation*}
	\left|\begin{matrix}
     x & y & z\\ 
     0 & 1 & 0\\
     a^2 & b^2 & c^2 
\end{matrix}\right| =0 \quad \Leftrightarrow \quad c^2x-a^2z=0.
\end{equation*}

The equation of the circumcircle of the reference triangle $ABC$ is well known [12] to be 
\begin{equation*}
	a^2qr + b^2pr + c^2pq = 0.
\end{equation*}
Therefore, the coordinates of the point $D$ are obtained by solving the system
\begin{equation*}
	\left\{ 
    \begin{array}{l}
        a^2qr + b^2pr + c^2pq = 0 \\
        c^2x-a^2z=0
    \end{array} 
\right.
\end{equation*}
from which it is easily found that the coordinates of $D$ are $\left(2 a^2: -b^2: 2 c^2\right)$.
\end{proof}

\begin{open}
Are there purely geometrical proofs for the results found in this section involving centers $X_3$, $X_4$, $X_5$, and $X_6$?
\end{open}


\section{Cyclic Orthodiagonal Quadrilaterals}

An \emph{cyclic orthodiagonal quadrilateral} is a cyclic quadrilateral whose diagonals are perpendicular.

Our computer study found a few relationships between a cyclic orthodiagonal quadrilateral
and its central quadrilateral (using any of the first 1000 centers) that were not true for cyclic quadrilaterals in general
or for orthodiagonal quadrilaterals in general.
These are shown in the following table.
\medskip

\begin{center}
\begin{tabular}{|l|p{2.2in}|}
\hline
\multicolumn{2}{|c|}{\textbf{\color{blue}\large \strut Central Quads of Cyclic Orthodiagonal Quadrilaterals}}\\ \hline
\ru $[ABCD]=[EHGF]$& 68\\ \hline
\ru $\anti[ABCD]=\diag[ABCD]=\stein[EFGH]$& 51--53, 128--130, 137--139, 143\\ \hline
\ru $\centro[ABCD]=\stein[EFGH]$& 568\\ \hline
\ru $\anti[ABCD]=\diag(ABCD)=\diag(EFGH)$& 6, 24, 25, 68, 186, 378, 847, 933\\ \hline
\ru $\m[ABCD]=\stein[EFGH]$& 389\\ \hline
\ru $\m[ABCD]=\diag[EFGH]$& 182, 216, 343\\ \hline
\ru $\persp[ABCD,GHEF]$& 186, 378, 571\\ \hline
\ru QA-P9=$\persp[ABCD,GHEF]$& 24\\ \hline
\end{tabular}
\end{center}

\newpage

\section{Kites}

A \emph{kite} is a quadrilateral consisting of two adjacent sides of length $a$ and the other two sides of length $b$.
A kite is necessarily orthodiagonal. 

Our computer study found a few relationships between a kite
and its central quadrilateral (using any of the first 1000 centers) that were not true for orthodiagonal quadrilaterals in general.
\medskip

\begin{center}
\begin{tabular}{|l|p{2.2in}|}
\hline
\multicolumn{2}{|c|}{\textbf{\color{blue}\large \strut Central Quadrilaterals of Kites}}\\ \hline
\ru $\m[ABCD]=\stein[EFGH]$& 402, 618--620\\ \hline
\ru $\ponce[ABCD]=\stein[EFGH]$& 13, 14\\ \hline
\ru $\stein[ABCD]=\stein[EFGH]$& 616, 617\\ \hline
\end{tabular}
\end{center}

\bigskip

We can assume without loss of generality that $D$ is the reflection of $B$ about $AC$. 
Hence, the barycentric coordinates of $D$ are $(a^2 + b^2 - c^2 : -b^2 : -a^2 + b^2 + c^2)$.

\medskip
When proving these results analytically using barycentric coordinates, we use the following result.

\begin{theorem}
If $ABCD$ is a kite with $AB=AD$ and $CB=CD$, then the Steiner point of its central quadrilateral coincides with the midpoint of $EG$.
\end{theorem}

\begin{figure}[h!t]
\centering
\includegraphics[width=0.25\linewidth]{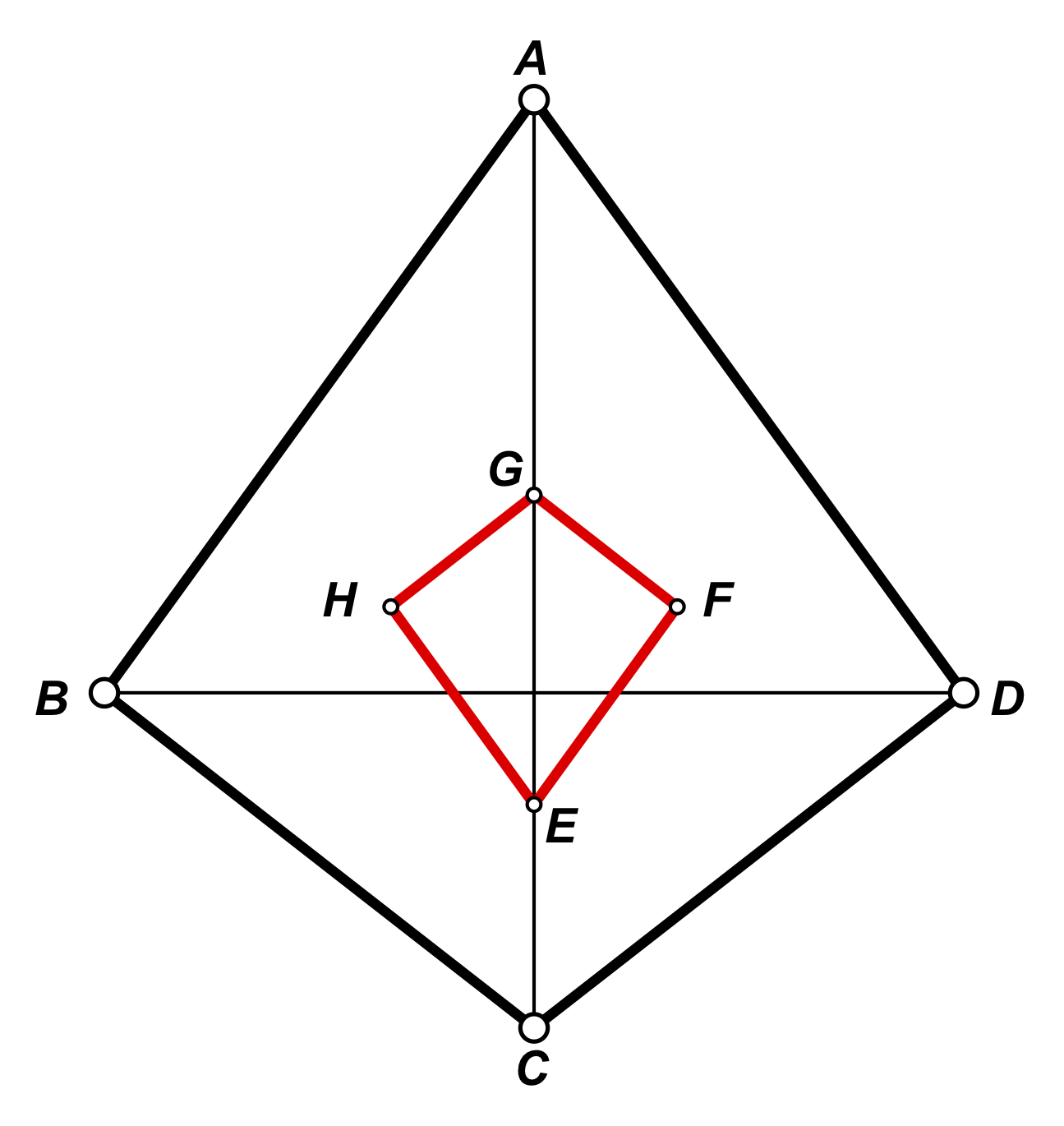}
\caption{kite $\implies$ kite}
\label{fig:halfKite}
\end{figure}

\begin{proof}
From Theorem~6.27 of \cite{shapes}, the central quadrilateral $EFGH$ is a kite with $EF=EH$ and $GF=GH$ (Figure~\ref{fig:halfKite}).
Then, by Corollary~10.5 of \cite{relationships}, the Steiner point of $EFGH$ coincides with the midpoint of $EG$.
\end{proof}


\section{AP Quadrilaterals}

An \emph{AP quadrilateral} is a quadrilateral whose sides (in order) form an arithmetic progression.

Our computer study found no relationships between an AP quadrilateral
and its central quadrilateral (using any of the first 1000 centers) that were not true for quadrilaterals in general.
\medskip

\begin{center}
\begin{tabular}{|l|p{2.2in}|}
\hline
\multicolumn{2}{|c|}{\textbf{\color{blue}\large \strut Central Quadrilaterals of AP Quadrilaterals}}\\ \hline
\multicolumn{2}{|c|}{No new relationships were found.}\\ \hline
\end{tabular}
\end{center}


\section{Equidiagonal Orthodiagonal Quadrilaterals}

An \emph{equidiagonal orthodiagonal quadrilateral} is a quadrilateral in which the two diagonals are both equal and perpendicular.

Our computer study found a few relationships between an equidiagonal orthodiagonal quadrilateral
and its central quadrilateral (using any of the first 1000 centers) that are not true for orthodiagonal quadrilaterals in general.
These are shown in the following table.
\medskip

\begin{table}[ht!]
\caption{}
\begin{center}
\begin{tabular}{|l|p{1.5in}|}
\hline
\multicolumn{2}{|c|}{\textbf{\color{blue}\large \strut Central Quads of Equidiagonal Orthodiagonal Quadrilaterals}}\\ \hline
\textbf{Relationship}&\textbf{centers}\\ \hline
\ru $\m[ABCD]=\m[EFGH]$&489\\ \hline
\ru $\stein[ABCD]=\stein[EFGH]$&638\\ \hline
\ru $\m[ABCD]=\stein[EFGH]$&640\\ \hline
\ru QA-P1=$\persp[ABCD,GHEF]$&485\\ \hline
\ru QA-P5=$\persp[ABCD,GHEF]$&68\\ \hline
\ru QA-P3=$\persp[ABCD,EFGH]$&637\\ \hline
\ru QA-P4=$\persp[ABCD,EFGH]$&489\\ \hline
\end{tabular}
\end{center}
\end{table}

\section{Exbicentric Quadrilaterals}

An \emph{exbicentric quadrilateral} is a cyclic quadrilateral that is also extangential.

Our computer study did not find any relationships between an exbicentric quadrilateral
and its central quadrilateral (using any of the first 1000 centers) that were not true for cyclic quadrilaterals in general.


\medskip

\begin{center}
\begin{tabular}{|l|p{2.2in}|}
\hline
\multicolumn{2}{|c|}{\textbf{\color{blue}\large \strut Central Quadrilaterals of Exbicentric Quadrilaterals}}\\ \hline
\multicolumn{2}{|c|}{No new relationships were found.}\\
\hline
\end{tabular}
\end{center}


\section{Parallelograms}

A \emph{parallelogram} is a quadrilateral in which both pairs of opposite sides are parallel.

Our computer study found hundreds of relationships between a parallelogram and its central quadrilateral.
Instead of listing all relationships found, we only list a few of the interesting relationships.

\medskip

\begin{center}
\begin{tabular}{|l|p{2.2in}|}
\hline
\multicolumn{2}{|c|}{\textbf{\color{blue}\large \strut Central Quadrilaterals of Parallelograms}}\\ \hline
\textbf{Relationship}&\textbf{centers}\\ \hline
\ru QA-P1=$\conic(ABCD,EFGH)$=QA-P1&7, 13, 14, 17, 18, 66, 330, 485, 486\\ \hline
\end{tabular}
\end{center}

\section{Bicentric Trapezoids}

A \emph{bicentric trapezoid} is a trapezoid that is also bicentric.

A bicentric trapezoid is necessarily an isosceles trapezoid.

Our computer study found a few relationships between a bicentric trapezoid
and its central quadrilateral (using any of the first 1000 centers) that are not true for bicentric quadrilaterals or isosceles trapezoids  in general.
These are shown in the following table.

\medskip

\begin{center}
\begin{tabular}{|l|p{3.2in}|}
\hline
\multicolumn{2}{|c|}{\textbf{\color{blue}\large \strut Central Quadrilaterals of Bicentric Trapezoids}}\\ \hline
\textbf{Relationship}&\textbf{centers}\\ \hline
\ru $\persp[ABCD,GHEF]$& 35, 36, 55, 56, 145, 999\\ \hline
\ru $\persp[ABCD,HGFE]$&49, 63, 92, 186, 265, 304, 305, 317, 328, 563\\ \hline
\end{tabular}
\end{center}


\bigskip

\void{
\begin{theorem}
Let $ABCD$ be a bicentric trapezoid.
Let $E$, $F$, $G$, and $H$ be the $X_{382}$-points of $\triangle BCD$, $\triangle CDA$, 
$\triangle DAB$, and $\triangle ABC$, respectively (Figure~\ref{fig:btX382}).
Then $ABCD$ and $EFGH$ are homothetic with ratio of similarity~$1/2$.
\end{theorem}

\begin{figure}[h!t]
\centering
\includegraphics[width=0.3\linewidth]{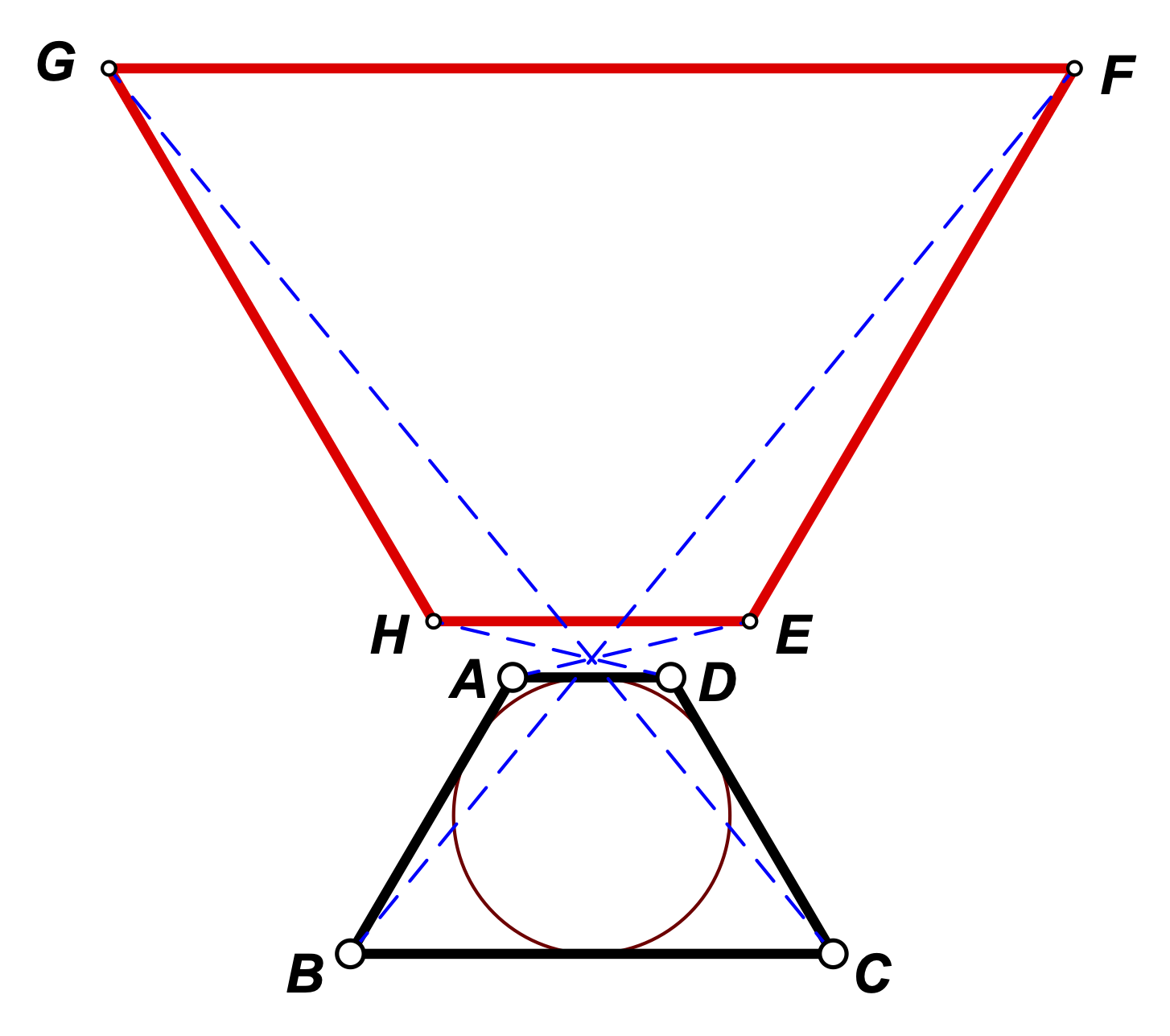}
\caption{Bicentric Trapezoid with $X_{382}$-points $\implies \ho[ABCD,EFGH]$}
\label{fig:btX382}
\end{figure}
}


\section{Rhombi}

A \emph{rhombus} is a quadrilateral all of whose sides have the same length.

Our computer study found hundreds of relationships between a rhombus and its central quadrilateral.
Instead of listing all relationships found, we only list a few of the interesting relationships.

\medskip

\begin{center}
\begin{tabular}{|p{2.4in}|p{1in}|}
\hline
\multicolumn{2}{|c|}{\textbf{\color{blue}\large \strut Central Quadrilaterals of Rhombi}}\\ \hline
\textbf{Relationship}&\textbf{centers}\\ \hline
\ru $[ABCD]=3[EFGH]$&13, 14\\ \hline
\ru $[ABCD]=4[EFGH]$&402, 620\\ \hline
\ru $[ABCD]=9[EFGH]$&290, 671, 903\\ \hline
\ru $[EFGH]=4[ABCD]$&446\\ \hline
\end{tabular}
\end{center}

\bigskip

\begin{open}
Are there purely geometrical proofs for the results found in this section involving centers $X_{13}$ and $X_{14}$?
\end{open}

\newpage

\section{Rectangles}

A \emph{rectangle} is a quadrilateral all of whose angles are right angles.

Our computer study found hundreds of relationships between a rectangle and its central quadrilateral.
Instead of listing all relationships found, we only list a few of the interesting relationships.

\medskip

\begin{center}
\begin{tabular}{|p{3in}|p{2.4in}|}
\hline
\multicolumn{2}{|c|}{\textbf{\color{blue}\large \strut Central Quadrilaterals of Rectangles}}\\ \hline
\textbf{Relationship}&\textbf{centers}\\ \hline
\ru $[ABCD]=25[EFGH]$&95\\ \hline
\ru $[ABCD]=2[EFGH]$&946\\ \hline
\ru $[ABCD]=4[EFGH]$&402\\ \hline
\ru $[EFGH]=9[ABCD]$&23\\ \hline
\ru $[ABCD]/[EFGH]=25/4$&233\\ \hline
\ru QA-P1=$\hyperb(ABCD,EFGH)$=QA-P1&251, 315, 481, 850, 961, 998\\ \hline
\ru $\partial[ABCD]=\partial[EFGH]$&46, 47, 117, 163, 579, 580, 920\\ \hline
\end{tabular}
\end{center}



\bigskip
When proving these results analytically using barycentric coordinates, we use the following result which is Theorem 6.32 in \cite{shapes}.

\begin{theorem}
If $ABCD$ be a rectangle, then the central quadrilateral is also a rectangle.
\end{theorem}

Figure~\ref{fig:reX46} shows the case when $n=46$.
Because of this theorem, to prove that the reference quadrilateral and the central quadrilateral have the same perimeter ($\partial[ABCD]=\partial[EFGH]$),
it is only necessary to prove that $AB+BC=EF+FG$.

\begin{figure}[h!t]
\centering
\includegraphics[width=0.35\linewidth]{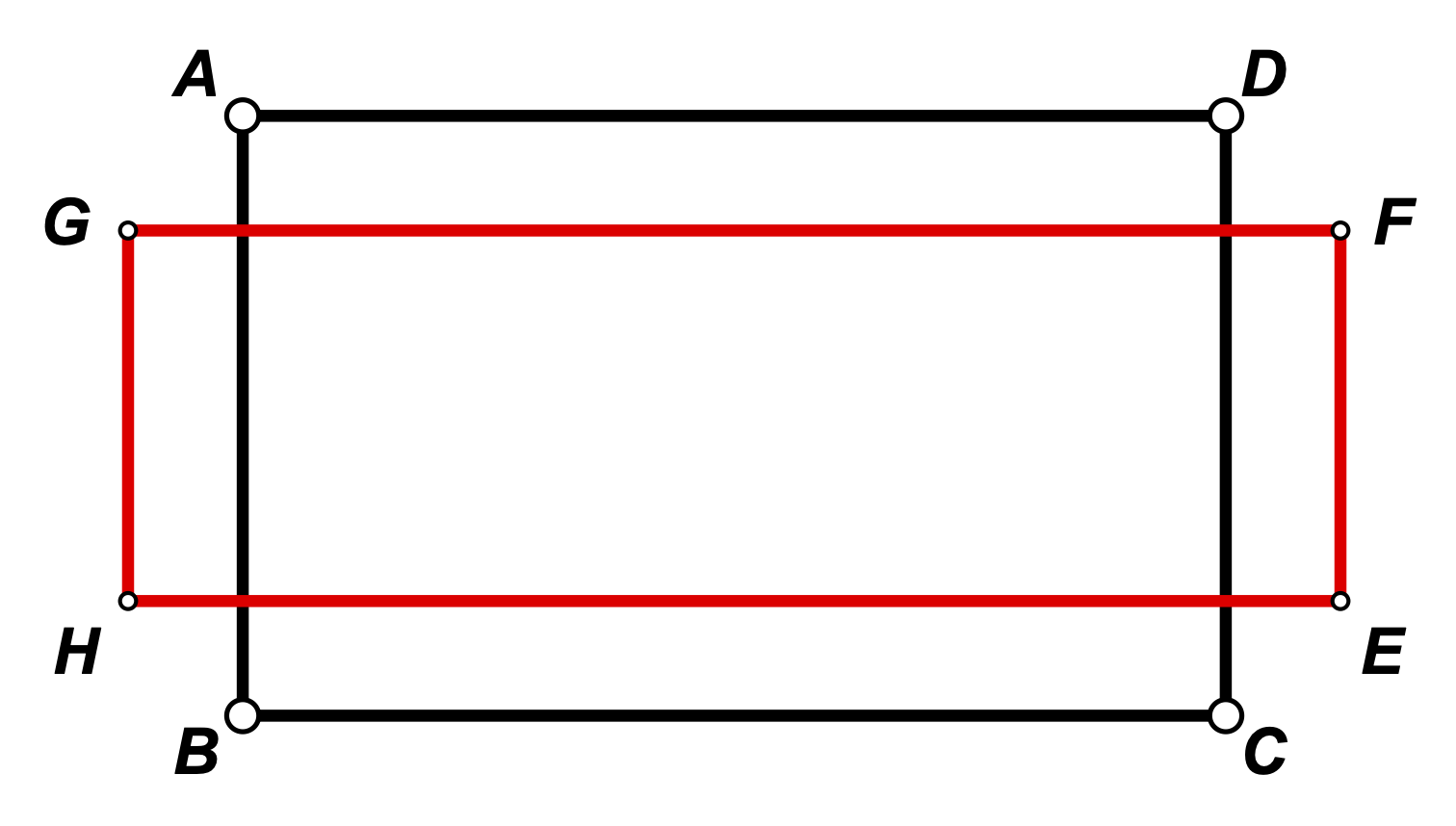}
\caption{Rectangle with $X_{46}$-points $\implies \partial[ABCD]=\partial[EFGH]$}
\label{fig:reX46}
\end{figure}

\newpage

\section{Squares}

A \emph{square} is a rectangle that is also a rhombus.

Our computer study found hundreds of relationships between a square and its central quadrilateral.
Instead of listing all relationships found, we only list a few of the interesting relationships.

\medskip

\begin{center}
\begin{tabular}{|p{3in}|p{2.4in}|}
\hline
\multicolumn{2}{|c|}{\textbf{\color{blue}\large \strut Central Quadrilaterals of Squares}}\\ \hline
\textbf{Relationship}&\textbf{centers}\\ \hline
\ru $[ABCD]=3[EFGH]$&13, 14\\ \hline
\ru $[ABCD]=49[EFGH]$&183, 252\\ \hline
\ru $[ABCD]=8[EFGH]$&496, 613, 988\\ \hline
\ru $[EFGH]=25[ABCD]$&352\\ \hline
\end{tabular}
\end{center}


\bigskip

\begin{open}
Are there purely geometrical proofs for the results found in this section involving centers $X_{13}$ and $X_{14}$?
\end{open}


\section{Areas for Future Research}

There are many avenues for future investigation.

\subsection{Investigate other triangle centers}\ \\

In our study, we only investigated triangle centers $X_n$ for $n\leq 1000$.
Extend this study to larger values of $n$.

As an example, the following result was found by Ercole Suppa \cite{Suppa}.

\begin{theorem}
\label{thm:cyX1173}
Let $ABCD$ be a cyclic quadrilateral.
Let $E$, $F$, $G$, and $H$ be the $X_{1173}$-points of $\triangle BCD$, $\triangle CDA$, 
$\triangle DAB$, and $\triangle ABC$, respectively.
Then quadrilaterals $ABCD$ and $EFGH$ have a common circumconic (Figure~\ref{fig:cqX1173}).
\end{theorem}

\begin{figure}[h!t]
\centering
\includegraphics[width=0.5\linewidth]{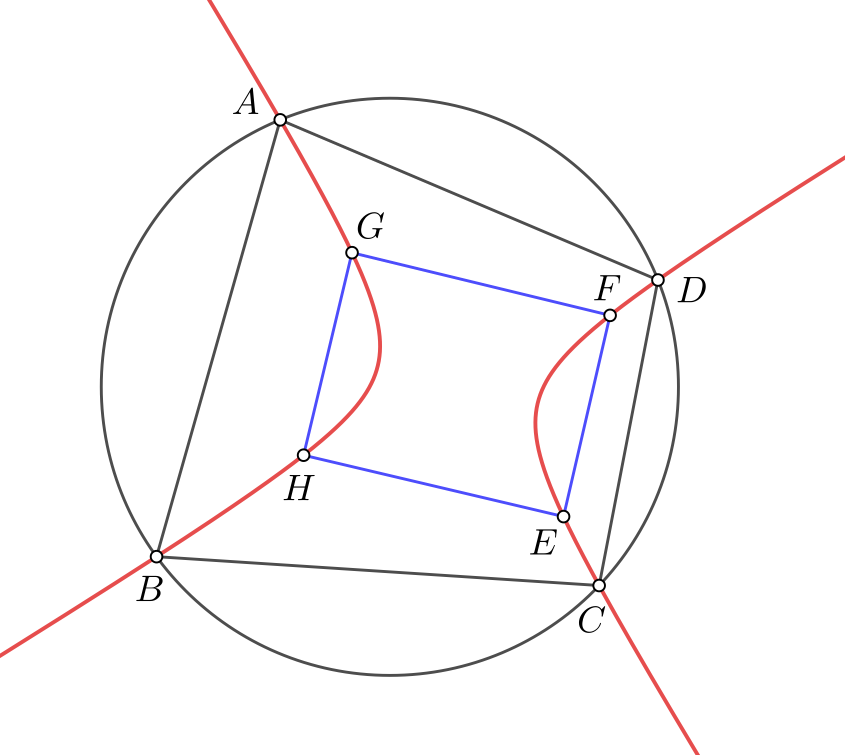}
\caption{cyclic quadrilateral with $X_{1173}$-points $\implies \conic[ABCD,EFGH]$}
\label{fig:cqX1173}
\end{figure}


\subsection{Use other shape quadrilaterals}\ \\

In our investigation, we only studied 28 shapes of quadrilaterals as shown in Figure~\ref{fig:quadShapes}.
There are many other shapes of quadrilaterals. Study these other shapes. For example, 
we say that a quadrilateral is \emph{orthoptic} if its opposite sides are perpendicular. Figure~\ref{fig:ooX6} shows an orthoptic quadrilateral in which
$AB\perp CD$ and $BC\perp AD$.
The following result was found by computer.

\begin{theorem}
\label{thm:ooX6}
Let $ABCD$ be an orthoptic quadrilateral.
Let $E$, $F$, $G$, and $H$ be the symmedian points ($X_{6}$-points) of $\triangle BCD$, $\triangle CDA$, 
$\triangle DAB$, and $\triangle ABC$, respectively.
Then quadrilaterals $ABCD$ and $EFGH$ are perspective.
The perspector is the Euler-Poncelet point (QA-P2) of quadrilateral $ABCD$.
\end{theorem}

\begin{figure}[h!t]
\centering
\includegraphics[width=0.5\linewidth]{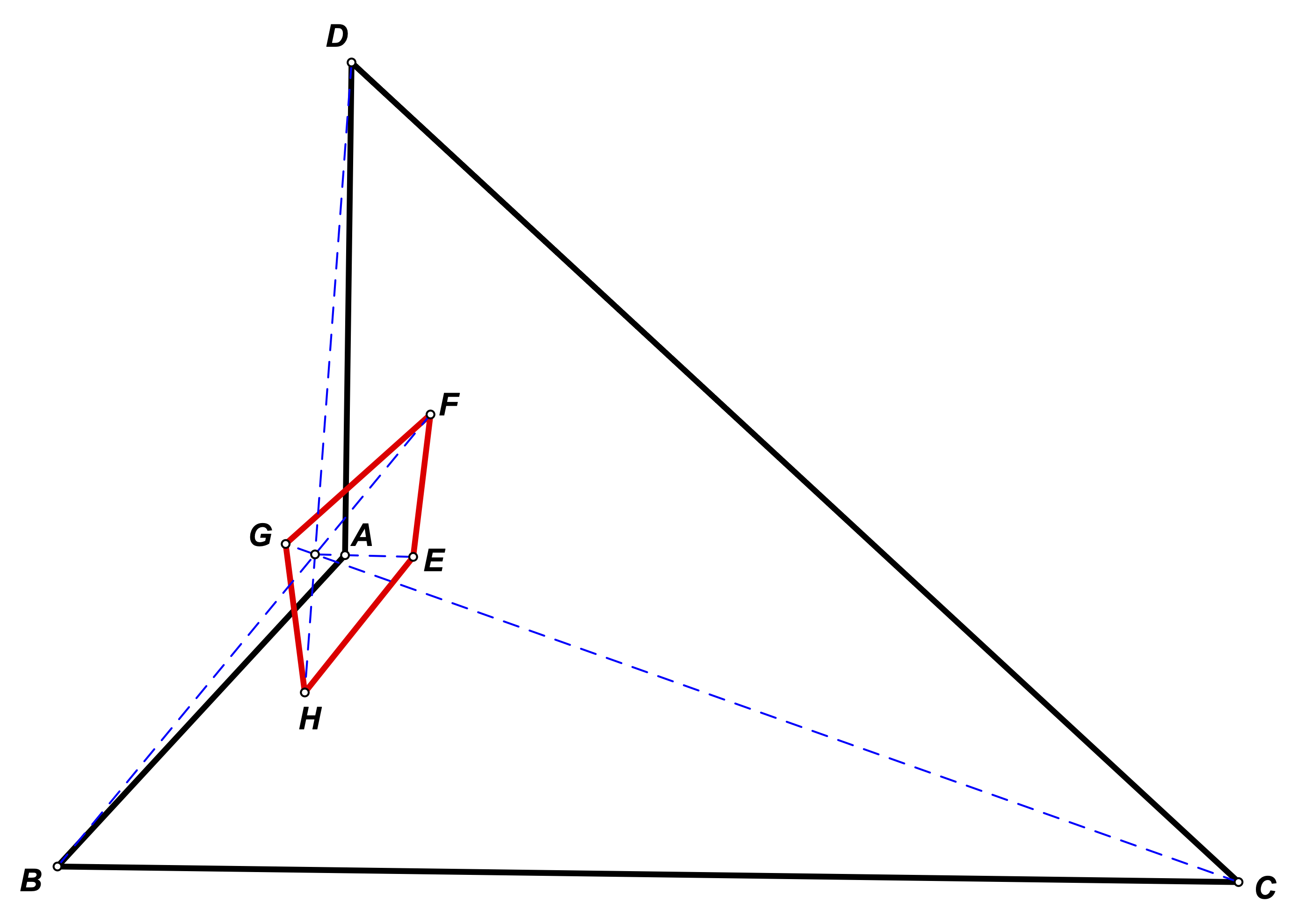}
\caption{orthoptic quadrilateral with $X_6$-points $\implies \persp[ABCD,EFGH]$}
\label{fig:ooX6}
\end{figure}

\begin{open}
Is there a purely geometrical proof of this result?
\end{open}

An \emph{orthocentric} quadrilateral is a quadrilateral in which each vertex is the orthocenter of the triangle
formed by the other three vertices.
The following result was found by computer.

\begin{theorem}
\label{thm:ocX11}
Let $ABCD$ be an orthocentric quadrilateral.
Let $E$, $F$, $G$, and $H$ be the Feuerbach points ($X_{11}$-points) of $\triangle BCD$, $\triangle CDA$, 
$\triangle DAB$, and $\triangle ABC$, respectively.
Then quadrilateral $EFGH$ is cyclic and the center of the circumcircle of $EFGH$
coincides with the centroid of quadrilateral $ABCD$ (Figure~\ref{fig:ocX11}).
\end{theorem}

\begin{figure}[h!t]
\centering
\includegraphics[width=0.45\linewidth]{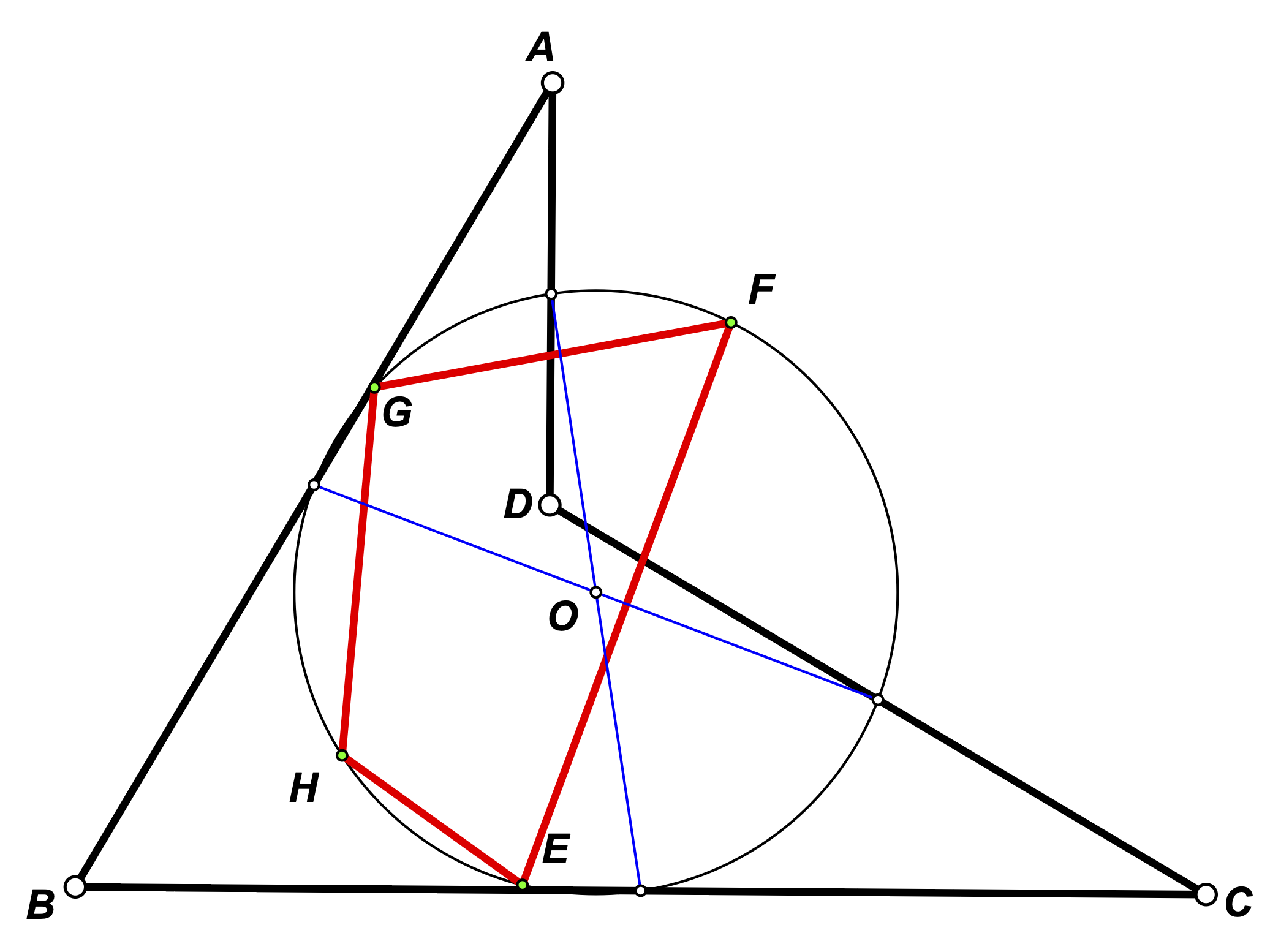}
\caption{orthocentric quad with $X_{11}$-points $\implies \m[ABCD]=\oo[EFGH]$}
\label{fig:ocX11}
\end{figure}

\begin{open}
Is there a purely geometrical proof of this result?
\end{open}

\subsection{Check for other quadrilateral centers}\ \\

In our study, when we checked to see if some center of the central quadrilateral coincides
with some center of the reference quadrilateral, we only checked the common centers listed in Table~\ref{table:centers}.
Additional centers could be investigated, such as the Miquel point (QL-P1), the area centroid (QG-P4),
the Morley Point (QL-P2), the Newton Steiner point (QL-P7), and the various quasi points.

\subsection{Investigate centers lying on quadrilateral lines}\ \\

We could also check to see if some center of the central quadrilateral lies on some notable line of the reference triangle,
such as the Newton line (QL-L1), the Steiner line (QL-L2), etc., or, in the case of cyclic quadrilaterals, the Euler line.

\subsection{Examine other properties}\ \\

There are many other properties between two quadrilaterals that can be studied.

For example, two polygons are \emph{orthogonal} if their corresponding sides are perpendicular.

\bigskip
The following result was found by computer.

\begin{theorem}
\label{thm:trX3}
Let $ABCD$ be a trapezoid.
Let $E$, $F$, $G$, $H$ be the $X_{3}$-points of $\triangle BCD$, $\triangle CDA$, 
$\triangle DAB$, and $\triangle ABC$, respectively.
Then quadrilaterals $ABCD$ and $HGFE$ are orthogonal (Figure~\ref{fig:trX3}).
\end{theorem}

\begin{figure}[h!t]
\centering
\includegraphics[width=0.35\linewidth]{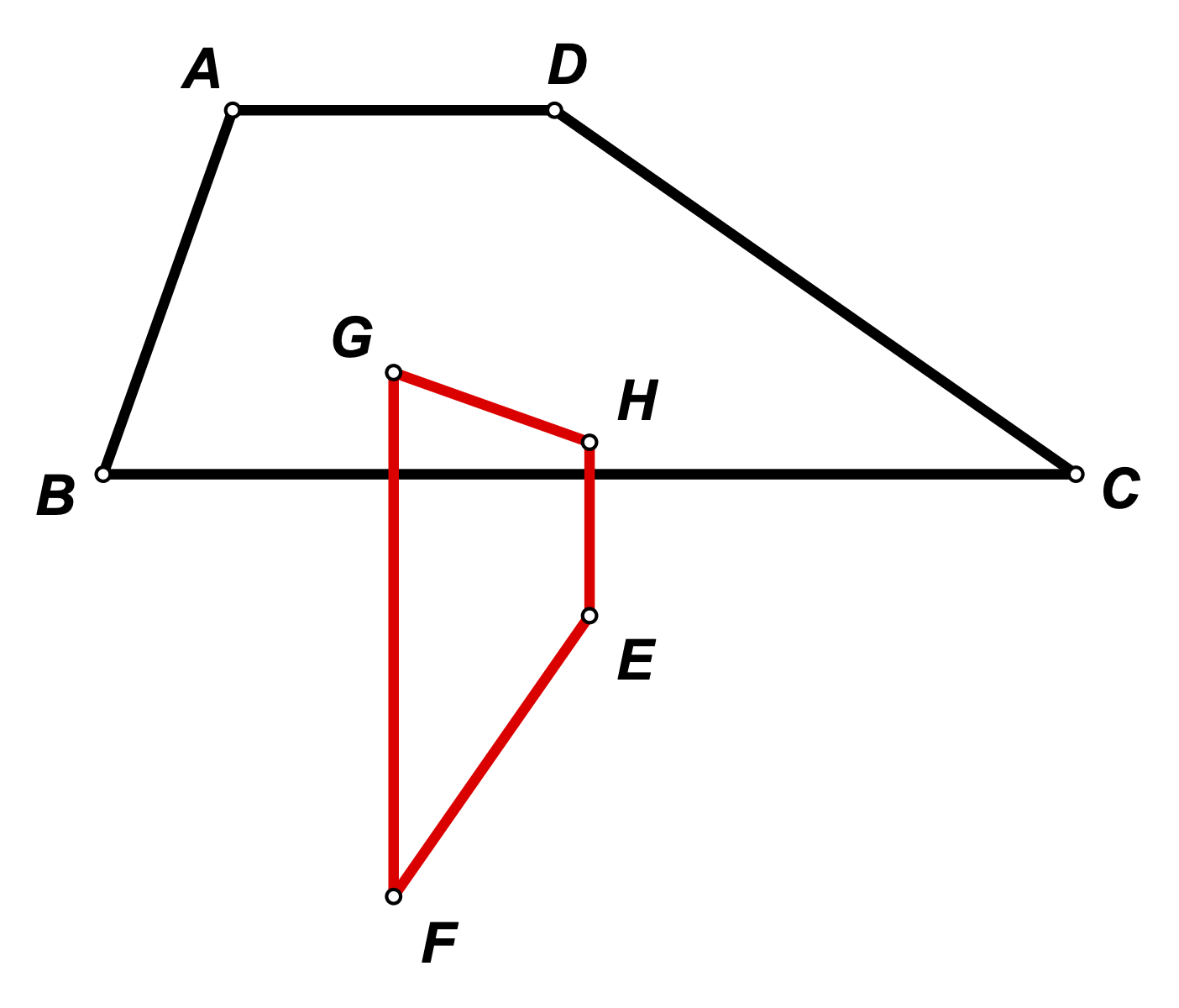}
\caption{trapezoid with $X_3$-points $\implies \mathrm{ortho}(ABCD,EFGH)$}
\label{fig:trX3}
\end{figure}

\begin{open}
Is there a center, $X$, such that quadrilaterals $ABCD$ and $EFGH$ have a common inconic?
\end{open}

\begin{open}
Is there a center, $X$, and a tangential quadrilateral $ABCD$, such that the central quadrilateral $EFGH$ formed with $X$-points is also
tangential and $ABCD$ and $EFGH$ have a common incircle? What about concentric incircles?
\end{open}

\goodbreak
\subsection{Investigate patterns in the center functions}\ \\

Many properties found are true for triangle centers $X_n$ for a list of values for~$n$.
What significance do these values have? Specifically, investigate the center functions associated
with these centers to see if some pattern can be found.

For example, it has been found that if $ABCD$ is cyclic, then $ABCD$ and $EFGH$ have a common non-circular circumconic for centers $X_n$
when $n=$4, 6, 54, 64, 1173, 11738, 3426, 3431, 11270, 13472, 13603, 14483, 14487, 14490, 14528, 16835, 11738...
What is the significance of these values of $n$?

Dylan Wyrzykowski \cite{Dylan2} has found the pattern with the following theorem.

\begin{theorem}
\label{thm:cyXn}
Let $ABCD$ be a cyclic quadrilateral.
Let $E$, $F$, $G$, and $H$ be the $X_{n}$-points of $\triangle BCD$, $\triangle CDA$, 
$\triangle DAB$, and $\triangle ABC$, respectively, where the isogonal conjugate of $X_n$
lies on the Euler line and has constant Shinagawa coefficients.
Then quadrilaterals $ABCD$ and $EFGH$ have a common circumconic.
\end{theorem}

We found in Section~\ref{section:cyclic}, that in a cyclic quadrilateral, $ABCD$ and $EFGH$ are homothetic
for centers $X_n$ when $n=$2, 4, 5, 20, 140, 376, 381, 382, 546- 550, 631, and 632.
What is the pattern giving rise to these values of $n$?

We found the following result.

\begin{theorem}
\label{thm:cyCos}
Let $ABCD$ be a cyclic quadrilateral.
Let $X$ be a triangle center whose (trilinear) center function is of the form $\cos B\cos C+k\cos A$,
where $k$ is some constant, not necessarily an integer.
Let $E$, $F$, $G$, and $H$ be the $X$-points of $\triangle BCD$, $\triangle CDA$, 
$\triangle DAB$, and $\triangle ABC$, respectively.
Then quadrilaterals $ABCD$ and $EFGH$ are homothetic.
\end{theorem}

\textbf{Note.} These are the points on the Euler line that have constant Shinagawa coefficients.

For similar results involving these points, see \cite{shapes}.

\bigskip
We found in Section~\ref{section:harmonic}, that if quadrilateral $ABCD$ is harmonic, then using centers $X_n$, we have
$\persp[ABCD,GHEF]$ when $n$=15, 16, 61, 62, 371, and 372. What is the pattern in these numbers? We found the following result by computer.

\begin{theorem}
\label{thm:hqPattern}
Let $ABCD$ be a harmonic quadrilateral.
Let $X$ be a triangle center whose center function is of the form $a\left(k(a^2-b^2-c^2)-S\right)$,
where $k$ is some constant and $S=2[ABC]$. 
Let $E$, $F$, $G$, and $H$ be the $X$-points of $\triangle BCD$, $\triangle CDA$, 
$\triangle DAB$, and $\triangle ABC$, respectively.
Then quadrilaterals $ABCD$ and $GHEF$ are perspective.
\end{theorem}

\bigskip
Another result, found by computer, involving a set of centers meeting a pattern is the following.

\bigskip
A \emph{power point} of a triangle is a triangle center whose center function is of the form $f(a,b,c)=a^k$,
where $k$ is a constant (not necessarily an integer).

\begin{theorem}
\label{thm:paXpow}
Let $ABCD$ be a parallelogram. Let $X$ be some power point of a triangle.
Let $E$, $F$, $G$, and $H$ be the $X$-points of $\triangle BCD$, $\triangle CDA$, 
$\triangle DAB$, and $\triangle ABC$, respectively.
Then quadrilaterals $ABCD$ and $EFGH$ have a common diagonal point.
\end{theorem}

\subsection{Ask about uniqueness}

Find an entry in one of our tables where there is only one center giving a particular relationship
for a certain type of quadrilateral. For example, for a general quadrilateral, $m[ABCD]=m[EFGH]$
seems to be true only when $n=2$. Is this because we only searched the first 1000 values of $n$?
Expand the search and find other values of $n$ for which the relationship is true or prove
that the result is unique. For example, we can state the following.

\begin{conjecture}
Let $ABCD$ be an arbitrary quadrilateral.
Let $X$ denote a triangle center.
Let $E$, $F$, $G$, and $H$ be the $X$-points of $\triangle BCD$, $\triangle CDA$, 
$\triangle DAB$, and $\triangle ABC$, respectively.
Then the centroid of $ABCD$ coincides with the centroid of $EFGH$ if and only if $X=X_2$.
\end{conjecture}

\subsection{Use notable points that are not triangle centers}\ \\

There are other points associated with a triangle that are not triangle centers.
Look for properties when some of these points are used.
For example, the following result was found by computer.

\begin{theorem}
\label{thm:sqBrocard}
Let $ABCD$ be a square.
Let $E$, $F$, $G$, and $H$ be the first Brocard points of $\triangle BCD$, $\triangle CDA$, 
$\triangle DAB$, and $\triangle ABC$, respectively (Figure~\ref{fig:sqBrocard}).
Then
$$[ABCD]=5[EFGH].$$
\end{theorem}

\begin{figure}[h!t]
\centering
\includegraphics[width=0.3\linewidth]{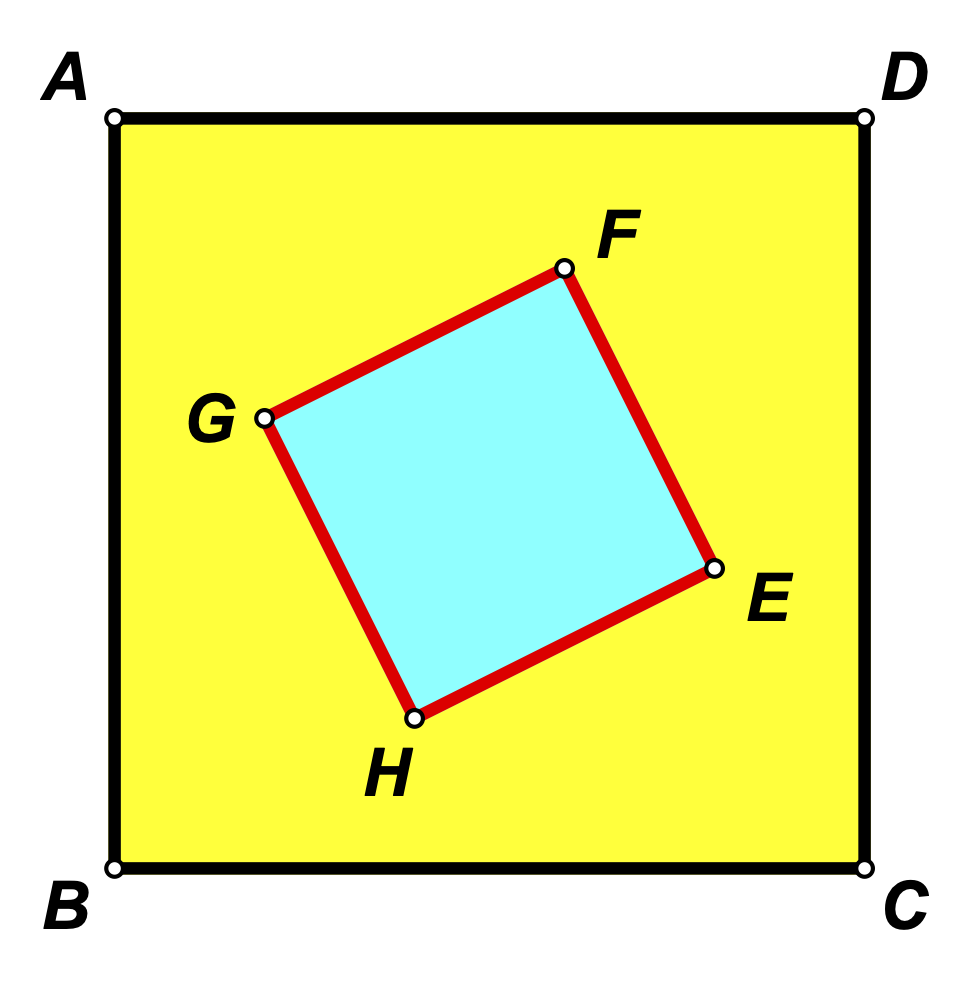}
\caption{square with Brocard points $\implies [ABCD]=5[EFGH]$}
\label{fig:sqBrocard}
\end{figure}

\begin{open}
Is there a purely geometrical proof of this result?
\end{open}

\subsection{Place different centers in different half triangles}\ \\

Would we find any interesting results if we place $X_n$-points in triangles $ABC$ and $ACD$, but place $X_m$
points in triangles $ABD$ and $BCD$, with $m\neq n$?

\subsection{Investigate some QL-properties}\ \\

If the lines $AE$, $BF$, $CG$, and $DH$ do not concur, then these four lines (with their points of intersection) form a figure known as
a \textit{complete quadrilateral}.
A complete quadrilateral has many notable points associated with it, such as the Miquel Point, the Morley Point,
the Clawson Center, and the Newton-Steiner Point. For a more extensive list see the section on Quadrilateral Points in \cite{EQF}.
Investigate whether any of these points coincide with notable points associated with the reference quadrilateral, $ABCD$.

\subsection{Work in 3-space}\ \\

If point $D$ is moved off the plane of $\triangle ABC$, then the reference quadrilateral becomes a tetrahedron
and the half triangles become the faces of the reference tetrahedron.
The central quadrilateral becomes the central tetrahedron. 
Investigate how the central tetrahedron is related to the reference tetrahedron.
Some results can be found in \cite{tetrahedron}.



\bigskip

\begin{thebibliography}{99}

\newcommand{\blue}{\color{blue}}
\newcommand{\black}{\color{black}}

\void{
\bibitem{Altshiller-Court}
Nathan Altshiller-Court,
\textit{College Geometry},
2nd edition.
Barnes \& Noble, Inc. NY: 1952.
}

\void{
\bibitem{Grozdev}
Sava Grozdev and Deko Dekov, \textit{Barycentric Coordinates: Formula Sheet},
International Journal of Computer Discovered Mathematics, \textbf{1}(2016)75--82.
\blue\url{http://www.journal-1.eu/2016-2/Grozdev-Dekov-Barycentric-Coordinates-pp.75-82.pdf}\black
}

\bibitem{KimberlingA}
Clark Kimberling,
\textit{Central Points and Central Lines in the Plane of a Triangle},
Mathematics Magazine, \textbf{67}(1994)163--187.\\
\blue\url{https://www.jstor.org/stable/2690608}\black

\void{
\bibitem{KimberlingB}
Clark Kimberling,
\textit{Triangle Centers and Central Triangles},
Congressus Numerantium, 129(1998)1--295.
}

\bibitem{ETC}
Clark Kimberling,
\textit{Encyclopedia of Triangle Centers}, 2025.\\
\blue\url{http://faculty.evansville.edu/ck6/encyclopedia/ETC.html}\black

\void{

\bibitem{ETC10}
Clark Kimberling,
X(10), \textit{Encyclopedia of Triangle Centers}, 2025.\\
\blue\url{http://faculty.evansville.edu/ck6/encyclopedia/ETC.html#X10}\black

\bibitem{ETC20}
Clark Kimberling,
X(20), \textit{Encyclopedia of Triangle Centers}, 2025.\\
\blue\url{http://faculty.evansville.edu/ck6/encyclopedia/ETC.html#X20}\black

\bibitem{ETC40}
Clark Kimberling,
X(40), \textit{Encyclopedia of Triangle Centers}, 2025.\\
\blue\url{http://faculty.evansville.edu/ck6/encyclopedia/ETC.html#X40}\black

\bibitem{ETC74}
Clark Kimberling,
X(74), \textit{Encyclopedia of Triangle Centers}, 2025.\\
\blue\url{http://faculty.evansville.edu/ck6/encyclopedia/ETC.html#X74}\black

\bibitem{ETC84}
Clark Kimberling,
X(84), \textit{Encyclopedia of Triangle Centers}, 2025.\\
\blue\url{http://faculty.evansville.edu/ck6/encyclopedia/ETC.html#X84}\black

\bibitem{ETC102}
Clark Kimberling,
X(102), \textit{Encyclopedia of Triangle Centers}, 2022.\\
\blue\url{http://faculty.evansville.edu/ck6/encyclopedia/ETC.html#X102}\black

\bibitem{ETC124}
Clark Kimberling,
X(124), \textit{Encyclopedia of Triangle Centers}, 2022.\\
\blue\url{http://faculty.evansville.edu/ck6/encyclopedia/ETC.html#X124}\black

\bibitem{ETC381}
Clark Kimberling,
X(381), \textit{Encyclopedia of Triangle Centers}, 2022.\\
\blue\url{http://faculty.evansville.edu/ck6/encyclopedia/ETC.html#X381}\black

\bibitem{ETC382}
Clark Kimberling,
X(382), \textit{Encyclopedia of Triangle Centers}, 2022.\\
\blue\url{http://faculty.evansville.edu/ck6/encyclopedia/ETC.html#X382}\black

\bibitem{ETC395}
Clark Kimberling,
X(395), \textit{Encyclopedia of Triangle Centers}, 2022.\\
\blue\url{http://faculty.evansville.edu/ck6/encyclopedia/ETC.html#X395}\black

\bibitem{ETC402}
Clark Kimberling,
X(402), \textit{Encyclopedia of Triangle Centers}, 2022.\\
{\blue\url{http://faculty.evansville.edu/ck6/encyclopedia/ETC.html#X402}\black}

\bibitem{ETC486}
Clark Kimberling,
X(486), \textit{Encyclopedia of Triangle Centers}, 2022.\\
{\blue\url{http://faculty.evansville.edu/ck6/encyclopedia/ETC.html#X486}\black}

\bibitem{ETC546}
Clark Kimberling,
X(546), \textit{Encyclopedia of Triangle Centers}, 2022.\\
{\blue\url{http://faculty.evansville.edu/ck6/encyclopedia/ETC.html#X546}\black}

\bibitem{ETC642}
Clark Kimberling,
X(642), \textit{Encyclopedia of Triangle Centers}, 2022.\\
{\blue\url{http://faculty.evansville.edu/ck6/encyclopedia/ETC.html#X642}\black}

}

\bibitem{tetrahedron}
Stanley Rabinowitz, \textit{Arrangement of Central Points on the Faces of a Tetrahedron},
International Journal of  Computer Discovered Mathematics. \textbf{5}(2020)13--41.\\
{\blue\url{https://arxiv.org/abs/2101.02592}\black}

\bibitem{shapes}
Stanley Rabinowitz and Ercole Suppa, \textit{The Shape of Central Quadrilaterals}.
International Journal of  Computer Discovered Mathematics. \textbf{7}(2022)131--180.\\
{\blue\url{https://arxiv.org/abs/2205.00870}\black}

\bibitem{relationships}
Stanley Rabinowitz and Ercole Suppa, \textit{Relationships between a Central Quadrilateral and its Reference Quadrilateral}.
International Journal of  Computer Discovered Mathematics. \textbf{7}(2022)214--287.\\
{\blue\url{https://arxiv.org/abs/2209.06008}\black}

\bibitem{Suppa}
Ercole Suppa, personal correspondence, Jan. 14, 2025.

\bibitem{EQF}
Chris van Tienhoven, Encyclopedia of Quadri-Figures.\\
{\blue\url{https://chrisvantienhoven.nl/mathematics/encyclopedia}\black}

\void{
\bibitem{QA-CT}
Chris van Tienhoven, \textit{Systematics for describing QA-points}. From Encyclopedia of Quadri-Figures.\\
{\blue\url{https://chrisvantienhoven.nl/qa-items/qa-geninf/qa-1}\black}
}

\bibitem{QA-P1}
Chris van Tienhoven, \textit{QA-P1: Quadrangle Centroid}. From Encyclopedia of Quadri-Figures.\\
{\blue\url{https://www.chrisvantienhoven.nl/qa-items/qa-points/qa-p1}\black}

\bibitem{QA-P2}
Chris van Tienhoven, \textit{QA-P2: Euler-Poncelet Point}. From Encyclopedia of Quadri-Figures.\\
{\blue\url{https://www.chrisvantienhoven.nl/qa-items/qa-points/qa-p2}\black}

\bibitem{QA-P34}
Chris van Tienhoven, \textit{QA-P34: Euler-Poncelet Point of the Centroid Quadrangle}. From Encyclopedia of Quadri-Figures.\\
{\blue\url{https://www.chrisvantienhoven.nl/qa-items/qa-points/qa-p34}\black}

\bibitem{MathWorld-Circumcircle}
Eric W. Weisstein, \textit{Circumcircle}. From MathWorld--A Wolfram Web Resource.\\
{\blue\url{https://mathworld.wolfram.com/Circumcircle.html}\black}

\bibitem{Wiki-Harmonic}
Wikipedia contributors, \textit{Harmonic Quadrilateral}. In Wikipedia, The Free Encyclopedia.\\
{\blue\url{https://en.wikipedia.org/w/index.php?title=Harmonic_quadrilateral&oldid=1260983024}\black}

\void{
\bibitem{Dylan}
Dylan Wyrzykowski, Problem in Romantics of Geometry Facebook Group, Jan. 6, 2025.
{\blue\url{https://www.facebook.com/groups/parmenides52/posts/9041609412619351/}\black}
}

\bibitem{Dylan2}
Dylan Wyrzykowski, \textit{On a Family of Circumconics}, submitted to International Journal of  Computer Discovered Mathematics, 2025.

\end{thebibliography}
\end{document}